\documentclass[12pt, a4paper, leqno]{amsart}

\usepackage{amsmath}
\usepackage{amsfonts}
\usepackage{amssymb}
\usepackage{mathrsfs}
\usepackage{color}
\usepackage[T1]{fontenc} %skandit
\usepackage{url}

\numberwithin{equation}{section}

\newtheorem{theorem}{Theorem}[section]
\newtheorem{definition}[theorem]{Definition}
\newtheorem{lemma}[theorem]{Lemma}
\newtheorem{proposition}[theorem]{Proposition}
\newtheorem{corollary}[theorem]{Corollary}

\theoremstyle{remark}
\newtheorem{remark}[theorem]{Remark}

\theoremstyle{plain}

\providecommand{\loc}{{\ensuremath{\mathrm{loc}}}}

\newcommand{\W}{\mathcal{W}^{\alpha}_{\alpha_1,\alpha_2}}
\newcommand{\w}{\textbf{\textit{w}}}

\newcommand{\px}{{p(\cdot)}}
\newcommand{\qx}{{q(\cdot)}}
\newcommand{\pzx}{{p_0(\cdot)}}
\newcommand{\pumx}{{p_1(\cdot)}}
\newcommand{\qzx}{{q_0(\cdot)}}
\newcommand{\qumx}{{q_1(\cdot)}}

\newcommand{\A}{A^{\textbf{\textit{w}}}_{\px,\qx}}
\newcommand{\B}{B^{\textbf{\textit{w}}}_{\px,\qx}}
\newcommand{\F}{F^{\textbf{\textit{w}}}_{\px,\qx}}

\newcommand{\as}{a^{\textbf{\textit{w}}}_{\px,\qx}}
\newcommand{\bs}{b^{\textbf{\textit{w}}}_{\px,\qx}}
\newcommand{\fs}{f^{\textbf{\textit{w}}}_{\px,\qx}}

\newcommand{\Nz}{\ensuremath{\mathbb{N}_0}}
\newcommand{\R}{\mathbb{R}}
\newcommand{\N}{\mathbb{N}}
\newcommand{\Z}{\mathbb{Z}}
\newcommand{\C}{\mathbb{C}}
\newcommand{\cS}{\mathcal{S}}

\newcommand{\Rn}{{\mathbb{R}^n}}
\newcommand{\Zn}{{\mathbb{Z}^n}}

\DeclareMathOperator{\supp}{supp}

\def\esssup{\operatornamewithlimits{ess\,sup}}
\def\essinf{\operatornamewithlimits{ess\,inf}}
\newcommand{\PPlog}{\mathcal{P}^{\log}(\Rn)}
\newcommand{\PP}{\mathcal{P}(\Rn)}

\begin{document}

\title[Atomic and molecular decompositions]{Atomic and molecular decompositions in variable exponent $2$-microlocal spaces and applications}

\author[A. Almeida]{Alexandre Almeida$^{*}$}
\address{Center for R\&D in Mathematics and Applications, Department of Mathematics, University of Aveiro, 3810-193 Aveiro, Portugal}
\email{jaralmeida@ua.pt}

\author[A. Caetano]{Ant\'{o}nio Caetano}
\address{Center for R\&D in Mathematics and Applications, Department of Mathematics, University of Aveiro, 3810-193 Aveiro, Portugal}
\email{acaetano@ua.pt}

\thanks{$^*$ Corresponding author.}
\thanks{This work was partially supported by Portuguese funds through CIDMA (Center for Research and Development in Mathematics and Applications) and FCT (Foundation for Science and Technology) within project UID/MAT/04106/2013.}
\thanks{\copyright 2015. Licensed under the CC-BY-NC-ND 4.0 license http://creativecommons.org/licenses/by-nc-nd/4.0/}
\thanks{Formal publication: http://dx.doi.org/10.1016/j.jfa.2015.11.010}

\date{\today}



\begin{abstract}
In this article we study atomic and molecular decompositions in $2$-microlocal Besov and Triebel--Lizorkin spaces with variable integrability. We show that, in most cases, the convergence implied in such decompositions holds not only in the distributions sense, but also in the function spaces themselves. As an application, we give a simple proof for the denseness of the Schwartz class in such spaces. Some other properties, like Sobolev embeddings, are also obtained via atomic representations.
\end{abstract}

\maketitle

%%%%%%%%%%%%%%%%%%%%%%%%%%%%%%%%%%%%%%%%%%%%%%%%%%%%%%%%
%%%%%%%%%%%%%%%%%%%%%%%%%%%%%%%%%%%%%%%%%%%%%%%%%%%%%%%%
%%%%%%%%%%%%%%%%%%%%%%%%%%%%%%%%%%%%%%%%%%%%%%%%%%%%%%%%

\section{Introduction}\label{sec:intro}

In this paper we deal with atomic and molecular decompositions for $2$-microlocal spaces of Besov and Triebel-Lizorkin type with all exponents variable, including applications, for example, to Sobolev type embeddings. This work can be viewed as a continuation of our paper \cite{AlmC15a}, where key properties like characterizations by Peetre maximal functions, liftings and Fourier multipliers were studied.

We also refer to \cite{AlmC15a} (and the references therein) for a review on the scales $\B$ and $\F$ helping to contextualize our study. We recall that the investigation of function spaces with variable exponents has been partially motivated by applications to fluid dynamics \cite{Ruz00}, image processing \cite{CheLR06,HarHLT13,Tiir14},
PDE and the calculus of variations \cite{AceM01,Fan07,LZhang13}; see also the monographs \cite{C-UF13,DHHR11} and the survey \cite{HarHLN10} for further details.

Atomic and molecular representations for the spaces $B^{\textbf{\textit{w}}}_{\px,q}$ and $\F$ (so with constant $q$ in the $B$ case) were already obtained by Kempka in \cite{Kem10}. In the present article we give characterizations in terms of atoms and molecules for the full scales above including the difficult case of variable $q$ in the Besov space $B^w_{\px,\qx}$ (see Section~\ref{sec:atom-mol-charact}) which, as we can see later, is far from being a mere extension. In fact, the mixed sequences spaces $\ell_\qx(L_\px)$ behind do not share some fundamental properties as in the constant exponent situation (like the boundedness of the Hardy-Littlewood maximal operator), and they are hard to deal with particularly when the exponent $q$ is unbounded.

We would like to emphasize that even in the cases studied in \cite{Kem10} our statements have different formulations. The idea is to improve and clarify some points and also to give additional information which is hard to find in the standard literature on atomic/molecular representations. By this reason we give some proofs and comments in a separate part (see Section~\ref{sec:proofs}).

Another of the main results (see Section~\ref{sec:converge-in-space}) asserts that in most cases the convergence implied in the atomic/molecular decompositions holds in the spaces themselves (see Theorem~\ref{thm:conv-space}). We show that this remarkable effect, not usually referred in the literature (but see \cite{Cae11} for an exception in the framework of classical spaces), may have interesting consequences (for example, to immediately get the denseness of the Schwartz class in the spaces $B^w_{\px,\qx}$ and $F^w_{\px,\qx}$ when $p$ and $q$ are bounded). Sobolev type embeddings are also established as application of the atomic decompositions obtained (see Section~\ref{sec:Sobolev-embed}).

In Sections~\ref{sec:prelim} and \ref{sec:microlocalspaces} we review some background material and derive some preliminary results needed in the sequel.

\section{Preliminaries}\label{sec:prelim}

As usual, we denote by $\mathbb{R}^{n}$ the $n$-dimensional real
Euclidean space, by $ \N$ the collection of all natural numbers and
$\N_{0}= \N\cup \{0\}$. By $\Zn$ we denote the lattice of all points in $\Rn$ with integer components. If $r$ is a real number then $r_+:=\max\{r,0\}$.
We write $B(x,r)$ for the open ball in
$\mathbb{R}^{n}$ centered at $x\in \mathbb{R}^{n}$ with radius $r>0$.
We use $c$ as a generic positive constant, i.e.\ a constant whose
value may change with each appearance. The expression $f
\lesssim g$ means that $f\leq c\,g$ for some independent constant
$c$, and $f\approx g$ means $f \lesssim g \lesssim f$.

The notation $X\hookrightarrow Y$ stands for continuous embeddings
from $X$ into $Y$, where $X$ and $Y$ are quasi-normed spaces. If
$E\subset {\mathbb{R}^{n}}$ is a  measurable set, then $|E|$ stands
for its (Lebesgue) measure and $\chi_{E}$ denotes its characteristic
function. By $\supp f$ we denote the support of the function $f$.

The set $\cS$ denotes the usual Schwartz class of infinitely differentiable
rapidly decreasing complex-valued functions and $\cS'$
denotes the dual space of tempered distributions. The Fourier
transform of a tempered distribution $f$ is denoted by $\hat f$ while its inverse transform is denoted by $f^\vee$.

\subsection{Variable exponents}

By $\PP$ we denote the set of all measurable functions $p:\Rn
\rightarrow (0,\infty]$ (called \textit{variable exponents}) which
are essentially bounded away from zero.  For $E\subset \Rn$
and $p\in \PP$ we denote $p_E^+ =\esssup_E p(x)$ and
$p_E^-=\essinf_E p(x)$. For simplicity we use the abbreviations $p^+=p_\Rn^+$ and
$p^-=p_\Rn^-$.

The \emph{variable exponent Lebesgue space} $L_\px=L_{\px}(\Rn)$ is the
class of all (complex or extended real-valued) measurable functions $f$ (on $\Rn$) such that
\[
\varrho_{\px}(f/\lambda):=\int_\Rn \phi_{p(x)}\left(\frac{|f(x)|}{\lambda}\right)\, dx
\]
is finite for some $\lambda>0$, where
\[
\phi_p(t) :=
\begin{cases}
t^p & \text{ if } p\in (0,\infty), \\
0 & \text{ if } p=\infty \text{ and } t\in [0,1], \\
\infty & \text{ if } p=\infty \text{ and } t\in(1,\infty]. \\
\end{cases}
\]

It is known that $\varrho_\px$ defines a semimodular (on the vector space consisting of all measurable functions on $\Rn$ which are finite a.e.), and that $L_{\px}$ becomes a quasi-Banach space with respect to the quasi-norm
\begin{align*}
    \| f|L_\px\| &:= \inf \left\{ \lambda>0 : \varrho_{\px}\left(f/\lambda\right) \leq 1\right\}.
\end{align*}
This functional defines a norm when $p^-\geq 1$.  If $p(x)\equiv p$ is constant, then $L_{\px}=L_{p}$ is the classical Lebesgue space.

It is worth noting that $L_{\px}$ has the lattice property and that the assertions $f\in L_\px$ and $\|f\,|\, L_\px\|<\infty$ are equivalent for any (complex or extended real-valued) measurable function $f$ (assuming the usual convention $\inf \varnothing=\infty$). The fundamental properties of the spaces $L_\px$, at least in the case $p^-\geq 1$, can be found in \cite{KR91} and in the recent monographs \cite{DHHR11}, \cite{C-UF13}. The definition above of $L_{\px}$ using the semimodular $\varrho_\px$ is taken from \cite{DHHR11}.

For any $p\in\PP$ we have
\[
\| f\,|\,L_\px\|^r = \left\| |f|^r | L_{\frac{\px}{r}}\right\|\,, \ \ \ \ r\in(0,\infty),
\]
and
\[
\| f+g\,|\,L_\px\| \leq \max\left\{1,2^{\frac{1}{p^-}-1}\right\} \left( \| f\,|\,L_\px\| + \| g\,|\,L_\px\|\right).
\]

A useful property (that we shall call the \textit{unit ball
property}) is that $\rho_\px(f) \leq 1$
if and only if $\| f\,|L_\px\| \leq 1$ (\cite[Lemma~3.2.4]{DHHR11}). An interesting variant of this is the following estimate
\begin{equation}\label{Lp-norm-mod}
\min \left\{ \varrho_{\px} (f)^{\frac{1}{p^-}}, \varrho_{\px} (f)^{\frac{1}{p^+}} \right\} \le \|f\,|\, L_\px\| \le \max \left\{ \varrho_{\px} (f)^{\frac{1}{p^-}}, \varrho_{\px} (f)^{\frac{1}{p^+}} \right\}
\end{equation}
for $p\in\PP$ with $p^-<\infty$, and $\varrho_{\px} (f)>0$ or $p^+<\infty$. It is proved in \cite[Lemma~3.2.5]{DHHR11} for the case $p^-\ge 1$, but it is not hard to check that this property remains valid in the case $p^- < 1$. This property is clear for constant exponents due to the
obvious relation between the quasi-norm and the semimodular in that case.

For variable exponents, H\"older's inequality holds in the form
\[
\| f\,g\,|L_1\| \leq 2\,\| f\,|L_\px\|  \| g\,|L_{p'(\cdot)}\|
\]
for $p\in\PP$ with $p^-\ge 1$, where $p'$ denotes the conjugate exponent of $p$ defined pointwisely by $\tfrac{1}{p(x)}+\tfrac{1}{p'(x)}=1, \ \ x\in\Rn$.

From the spaces $L_\px$ we can also define \emph{variable exponent Sobolev spaces} $W^{k,\px}$ in the usual way (see \cite{DHHR11}, \cite{C-UF13} and the references therein).

In general we need to assume some regularity on the exponents in order to develop a consistent theory of variable function spaces. We recall here some classes which are nowadays standard in this setting.

We say that a continuous function $g\,:\, \Rn\to \R$ is {\em locally $\log$-H\"{o}lder
continuous}, abbreviated $g \in C^{\log}_\loc$, if there exists $c_{\log}>0$ such that
  \begin{align*}%\label{ineq:log_cont}
    |g(x)-g(y)| \leq \frac{c_{\log}}{\log (e + 1/|x-y|)}
  \end{align*}
for all $x,y\in\Rn$. The function $g$ is said to be {\em globally $\log$-H\"{o}lder continuous},
abbreviated $g \in C^{\log}$, if it is locally $\log$-H\"{o}lder
continuous and there exist $g_\infty \in \R$ and $C_{\log}>0$ such that
\begin{align*}
%\label{ineq:log_decay}
|g(x) - g_\infty| &\leq \frac{C_{\log}}{\log(e
+ |x|)}
\end{align*}
for all $x \in \R^n$. The notation $\PPlog$ is used for those variable exponents $p\in \PP$ with $\frac1p \in C^{\log}$. We shall write $c_{\log}(g)$ when we need to use explicitly the constant involved in the local $\log$-Hölder continuity of $g$. Note that all (exponent) functions in $C^{\log}_{\loc}$ are bounded.

As regards the (quasi)norm of characteristic functions on cubes $Q$ (or balls) in $\Rn$, for exponents $p\in\PPlog$ we have
\begin{equation}\label{norm-charact-func-cubes}
\| \chi_Q\,| L_\px\| \approx |Q|^{\frac{1}{p(x)}} \ \ \ \ \textrm{if} \ \ \ |Q| \le 1 \ \ \ \textrm{and} \ \ x\in Q,
\end{equation}
and
\begin{equation*}
\| \chi_Q\,| L_\px\| \approx |Q|^{\frac{1}{p_\infty}} \ \ \ \ \textrm{if} \ \ |Q| \ge 1
\end{equation*}
(see \cite[Corollary~4.5.9]{DHHR11}).

\subsection{Variable exponent mixed sequence spaces}

To deal with variable exponent Besov and Triebel--Lizorkin scales we need to consider appropriate mixed sequences spaces. For $p,q\in\PP$ the \textit{mixed Lebesgue-sequence space} $L_\px(\ell_\qx)$ (\cite{DieHR09}) can be easily defined through the quasi-norm
\begin{equation} \label{def:lpq}
\|(f_\nu)_\nu\,| L_\px(\ell_\qx) \| := \big\|
\|(f_\nu(x))_\nu\,| \ell_{q(x)}\|\,| L_\px\big\|
\end{equation}
on sequences $(f_\nu)_{\nu\in\Nz}$ of complex or extended real-valued measurable functions on $\Rn$. This is a norm if $\min\{p^-,q^-\} \geq 1$.
Note that $\ell_{q(x)}$ is a standard discrete Lebesgue space (for each $x\in\Rn$), and that \eqref{def:lpq} is well defined since $q(x)$ does not depend on $\nu$ and the function $x\to \|(f_\nu(x))_\nu\,| \ell_{q(x)}\|$ is always measurable when $q\in\PP$.

The ``opposite'' case $\ell_\qx(L_\px)$ is not so easy to handle. For $p,q\in\PP$, the \textit{mixed sequence-Lebesgue space} $\ell_\qx(L_\px)$ consists of all sequences $(f_\nu)_{\nu\in\Nz}$ of (complex or extended real-valued) measurable functions (on $\Rn$) such that $\varrho_{\ell_\qx(L_\px)}\big( \tfrac{1}{\mu}(f_\nu)_\nu\big ) < \infty$ for some $\mu>0$, where
\begin{equation}\label{def:lpqmod}
\varrho_{\ell_\qx(L_\px)}\big( (f_\nu)_\nu\big ) := \sum_{\nu\ge 0}
\inf\Big\{\lambda_\nu>0: \, \varrho_{\px}\Big(f_\nu
/\lambda_\nu^{\frac1{\qx}}\Big)\le 1 \Big\}.
\end{equation}
Note that if $q^+<\infty$ then \eqref{def:lpqmod} equals the more simple form
\begin{equation}\label{def:lpqmodsimple}
\varrho_{\ell_\qx(L_\px)}\big( (f_\nu)_\nu\big ) = \sum_{\nu\ge 0} \Big\|
|f_\nu|^{\qx} | L_{\frac{\px}{\qx}}\Big\|.
\end{equation}
The space $\ell_\qx(L_\px)$ was introduced in \cite[Definition~3.1]{AlmH10} within the framework of the so-called \emph{semimodular spaces}.  It is known (\cite{AlmH10}) that
\begin{equation*}\label{def:lpqnorm}
  \|(f_\nu)_\nu\,| \ell_\qx(L_\px)\| :=
  \inf\Big\{ \mu>0:\, \varrho_{\ell_\qx(L_\px)}\big( \tfrac1\mu  (f_\nu)_\nu\big ) \le 1\Big\}
\end{equation*}
defines a quasi-norm in $\ell_\qx(L_\px)$ for every $p,q\in\PP$ and that $\|\cdot\,| \ell_\qx(L_\px)\|$ is a
norm when $q\geq 1$ is constant and $p^-\geq 1$, or when
$\frac{1}{p(x)}+\frac{1}{q(x)} \leq 1$ for all $x\in\Rn$. More recently, it was
shown in \cite{KemV13} that it also becomes a norm if $1\leq q(x)\leq p(x)\leq \infty$. Contrarily to the situation when $q$ is constant, the expression above is not necessarily a norm when $\min\{p^-,q^-\} \geq 1$ (see \cite{KemV13} for an example showing that the triangle inequality may fail in this case).

It is worth noting that $\ell_\qx(L_\px)$ is a really iterated space when $q\in(0,\infty]$ is constant (\cite[Proposition~3.3]{AlmH10}), with
\begin{equation}\label{iterated}
\|(f_\nu)_\nu\,| \ell_q(L_\px)\| = \big\| \big(\| f_\nu\,| L_\px\|\big)_\nu \,| \ell_q\big\|.
\end{equation}
We note also that the values of $q$ have no influence on $\|(f_\nu)_\nu\,| \ell_\qx(L_\px)\|$ when we consider sequences having just one non-zero entry. In fact, as in the constant exponent case, we have
$\|(f_\nu)_\nu\,|\, \ell_\qx(L_\px)\| = \|f\,|\, L_\px\|$ whenever there exists $\nu_0\in\Nz$ such that $f_{\nu_0}=f$ and $f_\nu \equiv 0$ for all $\nu \not= \nu_0$ (cf. \cite[Example~3.4]{AlmH10}).

Simple calculations show that given any sequence $(f_\nu)_\nu$ of measurable functions, finiteness of $\|(f_\nu)_\nu\,| \ell_\qx(L_\px)\|$ implies $(f_\nu)_\nu\in\ell_\qx(L_\px)$, which in turn implies $f_\nu\in L_\px$ for each $\nu\in\Nz$. Moreover,
$$\|(f_\nu)_\nu\,| \ell_\qx(L_\px)\| \le 1 \ \ \ \text{if and only if} \ \ \ \varrho_{\ell_\qx(L_\px)}\big( (f_\nu)_\nu\big ) \le 1 \ \ \ \ \ \text{(unit ball property)}$$
(see \cite{AlmH10}). From the latter, we can derive the following inequality on the estimation of the quasi-norm by the semimodular:

\begin{lemma}\label{lem:lqLp-norm-mod}
Let $p,q\in\PP$ with $q^-<\infty$. If $\varrho_{\ell_\qx(L_\px)} \big((f_\nu)_\nu\big)>0$ or $q^+<\infty$, then
\begin{equation*}%\label{lqLp-norm-mod}
\|(f_\nu)_\nu\,|\, \ell_\qx(L_\px)\| \le \max \left\{ \varrho_{\ell_\qx(L_\px)} \big((f_\nu)_\nu\big)^{\frac{1}{q^-}}, \varrho_{\ell_\qx(L_\px)} \big((f_\nu)_\nu\big)^{\frac{1}{q^+}} \right\}.
\end{equation*}
\end{lemma}

%\begin{proof}
%?????
%\end{proof}

The next result will be very useful below when checking the convergence implied in the atomic representations.

\begin{lemma}\label{lem:lpqconverge}
Let $p,q\in\PP$ with $q^+<\infty$. If $(f_\nu)_{\nu\in\Nz}\in\ell_\qx(L_\px)$, then
$$\left\|(f_\nu)_{\nu\ge T}\,|\ell_\qx(L_\px)\right\| \rightarrow 0 \ \  \ \ \text{as} \ \ T\to \infty.$$
\end{lemma}

\begin{proof}
From the previous lemma, the claim follows if we show that $\varrho_{\ell_\qx(L_\px)}\big( (f_\nu)_{\nu\ge T} \big ) \rightarrow 0$ as $T\to \infty$. Since $(f_\nu)_{\nu\in\Nz}\in\ell_\qx(L_\px)$, by definition there exists $\mu_0>0$ such that $\varrho_{\ell_\qx(L_\px)}\big( \mu_0 (f_\nu)_{\nu\in\Nz} \big )<\infty$. We have
\begin{equation*}
\varrho_{\ell_\qx(L_\px)}\big( (f_\nu)_{\nu\in\Nz} \big)  = \sum_{\nu =0}^\infty \Big\|
|f_\nu|^{\qx} | L_{\frac{\px}{\qx}}\Big\|
 \le  \max\big\{\big(\tfrac{1}{\mu_0}\big)^{q^+}, \big(\tfrac{1}{\mu_0}\big)^{q^-}\big\} \sum_{\nu =0}^\infty \Big\|
|\mu_0\,f_\nu|^{\qx} | L_{\frac{\px}{\qx}}\Big\|
\end{equation*}
using the fact that $q$ is bounded, and therefore the semimodular $\varrho_{\ell_\qx(L_\px)}$ takes the more simple form \eqref{def:lpqmodsimple} in that case. Since the right-hand side of this inequality is finite, we conclude that the numerical series
$$ \sum_{\nu =0}^\infty \Big\|
|f_\nu|^{\qx} | L_{\frac{\px}{\qx}}\Big\| $$
converges. Consequently
$$ \varrho_{\ell_\qx(L_\px)}\big( (f_\nu)_{\nu\ge T} \big) = \sum_{\nu =T}^\infty \Big\|
|f_\nu|^{\qx} | L_{\frac{\px}{\qx}}\Big\| \rightarrow 0 \ \ \ \text{as} \ \ T\to\infty. $$
\end{proof}

Both mixed sequence spaces $L_\px(\ell_\qx)$ and $\ell_\qx(L_\px)$ also satisfy the lattice property. Some basic embeddings involving these spaces are the following (cf. \cite[Lemma~2.1]{AlmC15a} and \cite[Theorem~6.1]{AlmH10}):

\begin{lemma}\label{lem:lpqembed}
Let $p,q,q_0,q_1\in\PP$. If $q_0\le q_1$ then we have
\begin{equation*}
L_\px(\ell_\qzx) \hookrightarrow L_\px(\ell_\qumx) \ \ \ \ \text{and} \ \ \ \ \ell_\qzx(L_\px) \hookrightarrow \ell_\qumx(L_\px).
\end{equation*}
Moreover, if $p^+,q^+<\infty$ then it also holds
\begin{equation*}
\ell_{\min\{\px,\qx\}}(L_\px)   \hookrightarrow L_\px(\ell_\qx) \hookrightarrow \ell_{\max\{\px,\qx\}}(L_\px).
\end{equation*}
\end{lemma}

Since the maximal operator does not behave well on the mixed spaces $L_\px(\ell_\qx)$ and $\ell_\qx(L_\px)$ (cf. \cite{AlmH10,DieHR09}), a key tool in this framework are the convolution inequalities below involving the functions
\[
\eta_{\nu,R}(x) := \frac{2^{n\nu}}{(1+2^\nu|x|)^R}\,, \ \ \ \nu\in \Nz, \ \ R>0.
\]
\begin{lemma}\label{lem:conv-eta}
Let $p,q\in\PPlog$ and $(f_\nu)_\nu$ be a sequence of non-negative measurable functions on $\Rn$.
\begin{enumerate}
\item[(i)] If $1<p^-\leq p^+ <\infty$ and $1<q^-\leq q^+ <\infty$, then for $R>n$ there holds
\[
\| (\eta_{\nu,R} * f_\nu)_\nu\,| L_\px(\ell_\qx) \| \lesssim \|
(f_\nu)_\nu \,| L_\px(\ell_\qx) \|.
\]
\item[(ii)] If $p^-\ge 1$ and $R>n+c_{\log}(1/q)$, then
\[
\| (\eta_{\nu,R} * f_\nu)_\nu\,| \ell_\qx(L_\px) \| \lesssim \|
(f_\nu)_\nu \,| \ell_\qx(L_\px) \|.
\]
\end{enumerate}
\end{lemma}

The convolution inequality in (i) above was given in \cite[Theorem~3.2]{DieHR09}, while the statement (ii) was established in \cite[Lemma~4.7]{AlmH10} and \cite[Lemma~10]{KemV12}.

\section{$2$-microlocal spaces with variable integrability}\label{sec:microlocalspaces}

Let $\alpha,\alpha_1,\alpha_2\in\R$ with $\alpha \ge 0$ and $\alpha_1 \le \alpha_2$.
We say that a sequence of positive measurable functions $\w = (w_j)_{j\in\Nz}$ belongs to class $\W$ if
\begin{enumerate}
\item[(i)] there exists $c>0$ such that
\begin{equation}\label{aux6}
0<w_j(x) \leq c\,w_j(y)\left(1+2^j|x-y|\right)^\alpha
\end{equation}
for all $j\in\Nz$ and $x,y\in\Rn$;
\item[(ii)] there holds
\[
2^{\alpha_1}\, w_j(x) \leq w_{j+1}(x) \leq 2^{\alpha_2}\, w_j(x)
\]
for all $j\in\Nz$ and $x\in\Rn$.
\end{enumerate}
Such a sequence will be called an \emph{admissible weight sequence}. When we write $\w\in\W$ without any restrictions it means that $\alpha \ge 0$ and $\alpha_1, \alpha_2\in\R$ (with $\alpha_1\leq \alpha_2$) are arbitrary but fixed numbers. We refer to \cite[Remark~2.4]{Kem08} for some useful properties of class $\W$, and to \cite[Examples~2.4--2.7]{AlmC15a} for a compilation of basic examples of admissible weight sequences.

We now recall the Fourier analytical approach to function spaces of Besov and Triebel-Lizorkin type. Let $(\varphi,\Phi)$ be a pair of functions in $\cS$ such that
\begin{itemize}
\item $\supp \hat{\varphi} \subset \{\xi\in\Rn: \, \tfrac{1}{2} \le |\xi| \le 2 \}\, $ and $\, |\hat{\varphi}(\xi)| \ge c>0$ when $\tfrac{3}{5} \le |\xi| \le \tfrac{5}{3}$;
\item $\supp \hat{\Phi} \subset \{\xi\in\Rn: \, |\xi| \le 2 \}\, $ and $\, |\hat{\Phi}(\xi)| \ge c>0$ when $ |\xi| \le \tfrac{5}{3}$.
\end{itemize}
Set $\varphi_0:=\Phi$ and $\varphi_j:=2^{jn}\varphi(2^j\cdot)$ for $j\in\N$. Then $\varphi_j\in\cS$ for all $j\in \Nz$ and
\[
\supp \widehat{\varphi}_j \subset \{\xi\in\Rn: \, 2^{j-1} \le |\xi| \le 2^{j+1} \} \,, \ \ \ j\in\N.
\]
Sometimes we call \emph{admissible} to a system $\{\varphi_j\}$ constructed in this way.

Given such a system, we define $2$-microlocal Besov and Triebel-Lizorkin spaces with variable integrability in the following way:

\begin{definition} %\label{def:BFspaces}
Let $\w=(w_j)_j\in\W$ and $p,q\in\PP$.
\begin{enumerate}
\item[(i)] $\B$ is the set of all $f\in\cS'$ such that
\begin{equation}\label{Bnorm}
\|f\,|\B\|:= \big\| (w_j(\varphi_j\ast f))_j\,|\, \ell_\qx(L_\px)\big\| < \infty.
\end{equation}
\item[(ii)] Restricting to $p^+,q^+<\infty$, $\F$ is the set of all $f\in\cS'$ such that
\begin{equation}\label{Fnorm}
\|f\,|\F\|:= \big\| (w_j(\varphi_j\ast f))_j\, |\, L_\px(\ell_\qx)\big\| < \infty.
\end{equation}
\end{enumerate}
\end{definition}

The sets $\B$ and $\F$ become quasi-normed spaces equipped with the quasi-norms \eqref{Bnorm} and \eqref{Fnorm}, respectively. Moreover, they are independent of the admissible pair $(\varphi,\Phi)$ taken in its definition, at least when $p$ and $q$ satisfy some regularity assumptions. Indeed, if $\w\in\W$ and $p,q\in\PPlog$ (with $\max\{p^+,q^+\}<\infty$ in the $F$-case), then different such pairs produce equivalent quasi-norms in the corresponding spaces. This fact can be obtained as a consequence of the Peetre maximal functions characterization. We refer to the papers \cite[Theorem~6]{KemV12}, \cite[Corollary~4.7]{Kem09}, and also to \cite[Theorem~3.1 and Corollary~3.2]{AlmC15a} where some clarification and additional information is given. Notice that the ``unnatural'' restriction $q^+<\infty$ appearing in the $F$ space comes essentially from the application of Lemma~\ref{lem:conv-eta}(i) (cf. \cite{DieHR09}).

For simplicity we will omit the reference to the admissible pair $(\varphi,\Phi)$ used to define the quasi-norms \eqref{Bnorm} and \eqref{Fnorm}. Moreover, we often write $\A$ for short when there is no need to distinguish between Besov and Triebel-Lizorkin spaces.

The unified treatment given by the two scales above relies on \cite{Kem09}. The spaces $\A$ include, as particular cases, the Besov and Triebel-Lizorkin spaces with variable smoothness and integrability, $B^{s(\cdot)}_{\px,\qx}$  and $F^{s(\cdot)}_{\px,\qx}$, introduced in \cite{AlmH10} and \cite{DieHR09}, respectively, which in turn contain variable order Hölder-Zygmund spaces (\cite[Section~7]{AlmH10}), and variable exponent Sobolev spaces and Bessel potential spaces (\cite{AlmS06,GHN}). Moreover, weighted function spaces (see \cite[Chapter~4]{EdTri96}) and spaces with generalized smoothness (\cite{FarL06,KalLiz87,Mou01}) are also included in the scales $\A$. Classically, a fundamental example of $2$-microlocal spaces (from which the terminology seems to come from) are the spaces constructed from the special weight sequence given by $w_j(x)= 2^{js}(1+2^j\,|x-x_0|)^{s'}$ $(s,s'\in\R)$, in connection with the study of regularity properties of functions (see \cite{Pee75,Bo84,Jaf91,JafMey96}). We refer to \cite{AlmC15a,Kem09} for further references and details.

One of the objectives of this paper is to derive Sobolev type embeddings using appropriate atomic representations. At a more basic level, in \cite[Section~5]{AlmC15a} we have shown that, for $p,q,q_0,q_1\in\PP$ and $\w\in\W$, we have
\begin{equation*}
A^{\w}_{\px,\qzx} \hookrightarrow A^{\w}_{\px,\qumx}
\end{equation*}
when $q_0\le q_1$ (with $p^+,q_0^+,q_1^+ < \infty$ when $A=F$) and
\begin{equation*}
B^\w_{\px,\min\{\px,\qx\}} \hookrightarrow
\F \hookrightarrow B^\w_{\px,\max\{\px,\qx\}}
\end{equation*}
if $p^+,q^+ < \infty$. In particular, $B^\w_{\px,\px} = F^\w_{\px,\px}$.  In \cite{AlmC15a} we have also obtained the following embeddings which will be useful later on:
\begin{proposition}\label{pro:embed1}
Let $\w\in\W$, $\textbf{v}\in\mathcal{W}^{\beta}_{\beta_1,\beta_2}$, $p,q_0,q_1\in\PP$ and $\tfrac{1}{q^\ast}:= \Big(\tfrac{1}{q^-_1}-\tfrac{1}{q^+_0}\Big)_+$.\\ If $\big(\tfrac{v_j}{w_j}\big)_j\in \ell_{q^\ast}(L_\infty)$ when $A=B$ or
$\big(\tfrac{v_j}{w_j}\big)_j\in L_\infty(\ell_{q^\ast})$ and $p^+,q_0^+,q_1^+ < \infty$ when $A=F$, then
\begin{equation}\label{embed1}
A^{\w}_{\px,\qzx} \hookrightarrow A^{\textbf{v}}_{\px,\qumx}.
\end{equation}
\end{proposition}

Although we are interested to work with function spaces independent of the starting system $\{ \varphi_j\}$, the fact is that we do not need the $\log$-H\"older conditions on the exponents to derive the previous embeddings. This means that the conditions assumed above are those actually needed in the proofs, and that the results should then be understood to hold when the same fixed system is used for the definition of all spaces involved.

We also have
\[
\cS \hookrightarrow \A \hookrightarrow \cS'
\]
for $p,q\in\PPlog$ and $\w\in\W$ (with $p^+,q^+ < \infty$ in the $F$ case). We refer to \cite[Theorem~5.4]{AlmC15a} for a proof using liftings. An alternative proof could also be derived using Sobolev type embeddings (see Corollary~\ref{cor:Sobolev-embed}) and the above embeddings, following the scheme used in \cite[Theorem~2.11]{MouNS13} for constant $q$.

\section{Atomic and molecular decompositions}\label{sec:atom-mol-charact}

We use the notation $Q_{jm}$, with $j\in\Nz$ and $m\in\Zn$, for the closed cube with sides parallel to the coordinate axes, centered at $2^{-j}m$ and with side length $2^{-j}$. By $\chi_{jm}$ we denote the corresponding characteristic function. The notation $d\,Q_{jm}$, $d>0$, will stand for the closed cube concentric with $Q_{jm}$ and of side length $d2^{-j}$.

\begin{definition}[Atoms]
Let $K,L\in\Nz$ and $d>1$. For each $j\in\Nz$ and $m\in\Zn$, a $C^K$-function $a_{jm}$ on $\Rn$ is called a $(K,L,d)$-atom (supported near $Q_{jm}$) if
\begin{equation*}
\supp a_{jm} \subset d\,Q_{jm}\,,
\end{equation*}
\begin{equation}\label{atom-supp}
\sup_{x\in\Rn} |D^\gamma a_{jm}(x)| \le 2^{|\gamma|j}\,, \ \ \ \ 0\le |\gamma| \le K,
\end{equation}
and
\begin{equation}\label{atom-moments}
\int_{\Rn} x^\gamma\, a_{jm}(x)\, dx=0\,, \ \ \ \ 0\le |\gamma| <L.
\end{equation}
\end{definition}

A $(K,0,d)$-atom is an atom for which the requirement \eqref{atom-moments} is void, that is, for which the so-called \emph{moment conditions} are not required. The $a_{0,m}$ atoms we are going to consider are always $(K,0,d)$-atoms.

\begin{definition}[Molecules]
Let $K,L\in\Nz$ and $M>0$. For each $j\in\Nz$ and $m\in\Zn$, a $C^K$-function $[aa]_{jm}$ on $\Rn$ is called a $(K,L,M)$-molecule (concentrated near $Q_{jm}$) if
\begin{equation}\label{mol-supp}
|D^\gamma [aa]_{jm}(x)| \le 2^{|\gamma|j} (1+2^j|x-2^{-j}m|)^{-M}\,, \ \ \ \ x\in\Rn, \ \ 0\le |\gamma| \le K,
\end{equation}
and
\begin{equation}\label{mol-moments}
\int_{\Rn} x^\gamma\, [aa]_{jm}(x)\, dx=0\,, \ \ \ \ 0\le |\gamma| <L.
\end{equation}
\end{definition}

A $(K,0,M)$-molecule is a molecule for which the requirement \eqref{mol-moments} is void (i.e., moment conditions are not required). The $[aa]_{0,m}$ molecules we will consider are always $(K,0,M)$-molecules.

\begin{remark}\label{atom-is-molecule}
It is easy to check that if $a_{jm}$ is a $(K,L,d)$-atom (so supported near $Q_{jm}$), then, given any $M>0$, $\big(1+\tfrac{d\sqrt{n}}{2}\big)^{-M}a_{jm}$ is a $(K,L,M)$-molecule concentrated near $Q_{jm}$.
\end{remark}

\begin{definition}
Let $\w\in\W$ and $p,q\in\PP$. The sets $\bs$ and $\fs$ consist of all (complex-valued) sequences $\lambda=(\lambda_{jm})_{j\in\Nz\;\; \atop m\in\Zn}$ such that
\begin{equation*}
\|\lambda \,| \bs\|:= \left\| \left( \sum_{m\in\Zn} \lambda_{jm}\,w_j(2^{-j}m)\, \chi_{jm} \right)_j\, | \ell_\qx(L_\px) \right\| < \infty
\end{equation*}
and
\begin{equation*}
\|\lambda \,| \fs\|:= \left\| \left( \sum_{m\in\Zn} \lambda_{jm}\,w_j(2^{-j}m) \, \chi_{jm} \right)_j\, | L_\px(\ell_\qx) \right\| < \infty,
\end{equation*}
respectively.
\end{definition}

Using standard arguments, it is easy to see that the sets $\bs$ and $\fs$ become quasi-normed spaces equipped with the functionals above. Furthermore, taking into account the properties of class $\W$, we can use $w_j(x)$ in place of $w_j(2^{-j}m)$ in the expressions above, up to equivalent quasi-norms.

For short we will write $(\lambda_{jm})$ instead of $(\lambda_{jm})_{j\in\Nz\;\; \atop m\in\Zn}$  when it is clear we are working with the exhibited set of indices. Similarly as before, we shall write $\as$ to refer to both spaces without distinction.

%If $p,q$ are constant and $w_j \equiv 2^{js}$, $s\in\R$, $j\in\Nz$, then $\as=a^s_{p,q}$ are the standard sequence spaces \textcolor{red}{(cf. ??????????).}

From Lemma~\ref{lem:lpqembed} we immediately get the following result:

\begin{corollary}\label{cor:embedbf}
Let $\w\in\W$ and $p,q,q_0,q_1\in\PPlog$. If $q_0 \le q_1$, then
\begin{equation*}
a^{\w}_{\px,\qzx} \hookrightarrow a^{\w}_{\px,\qumx}.
\end{equation*}
Moreover, if $p^+,q^+ <\infty$, then
\begin{equation*}
b^\w_{\px,\min\{\px,\qx\}} \hookrightarrow
\fs \hookrightarrow
b^\w_{\px,\max\{\px,\qx\}}.
\end{equation*}
\end{corollary}

For $r,t\in(0,\infty]$ we use the notation
$$ \sigma_{r,t}:=n\left(\tfrac{1}{\min\{1,r,t\}}-1 \right)  \ \ \ \text{and} \ \ \ \sigma_{r}:=\sigma_{r,r}.$$

Before passing to the main atomic and molecular representation statements, we would like to take the opportunity to clarify some convergence issues that typically are not so clearly mentioned in the literature.

%The following result is essentially Lemma~3.11 in \cite{Kem10}.

\begin{proposition}\label{pro:conv}
Let $\w\in\W$, $p\in\PPlog$. Let also $K,L\in\Nz$ and $M>0$ be such that
\begin{equation}\label{conv-conditions}
L>\sigma_{p^-}-\alpha_1 \ \ \ \ \ \text{and} \ \ \ \ \ M>L+2n+2\alpha+2\,c_{\log}(1/p)\,\sigma_{p^-}.
\end{equation}
Let $[aa]_{0m}$, $m\in\Zn$, be $(K,0,M)$-molecules and, for $j\in\N$ and $m\in\Zn$, $[aa]_{jm}$ be $(K,L,M)$-molecules. If $\lambda\in b^\w_{\px,\infty}$, then
\begin{equation}\label{conv}
\sum_{j=0}^\infty \sum_{m\in\Zn} \lambda_{jm}\, [aa]_{jm} %\ \ \ \ \text{converges in} \ \ \cS'.
\end{equation}
converges in $\cS'$ and the convergence in $\cS'$ of the inner sum gives the regular distribution obtained by taking the corresponding pointwise convergence. Moreover, the sum
\begin{equation}\label{conv-double}
\sum_{(j,m)\in\Nz\times\Zn} \lambda_{jm}\, [aa]_{jm}
\end{equation}
converges also in $\cS'$ (to the same distribution to which the iterated series in \eqref{conv} converges).
\end{proposition}

\begin{remark}
Although in \cite[Lemma~3.11]{Kem10} we can see an effort to clarify the convergence implied in \eqref{conv}, only the proof of the convergence of the outer sum is outlined, see \cite[Appendix]{Kem10}. So, in order to complete the picture, we refer the reader to Subsection~\ref{sec:proof-pro} below, where the missing parts are proved.
\end{remark}

%%\begin{remark}
%%\hspace{100mm}
%%\begin{enumerate}
%%\item[(i)] Although in \cite[Lemma~3.11]{Kem10} we can see an effort to clarify the convergence implied in \eqref{conv}, the point is that only the convergence of the outer sum is really carried out, see \cite[Appendix]{Kem10}. As regards the inner sum in \eqref{conv}, it is often treated in the literature with the pointwise sense only. As we show below, its convergence also holds in $\cS'$, being this fact important to justify some technical calculations involving distributions.
%%\item[(ii)] The previous proposition holds also with $[aa]_{jm}$ replaced by $(K,L,d)$-atoms $a_{jm}$ when $j\in\N$ and by $(K,0,d)$-atoms $a_{0m}$ when $j=0$, under the same hypotheses except those on $M$ (which are not required then). Indeed, as observed in Remark~\ref{atom-is-molecule}, given such atoms, for any $M>0$, $\big(1+\tfrac{d\sqrt{n}}{2}\big)^{-M}a_{jm}$ are $(K,L,M)$- and $(K,0,M)$-molecules, respectively. As long as we choose $M$ satisfying the second assumption in \eqref{conv-conditions} and take $\lambda\in \as$, we guarantee the convergence (in $\cS'$) of
%%    $$\sum_{j=0}^\infty \sum_{m\in\Zn} \lambda_{jm}\, \big(1+\tfrac{d\sqrt{n}}{2}\big)^{-M}a_{jm}$$
%%    (where clearly the constant $\big(1+\tfrac{d\sqrt{n}}{2}\big)^{-M}$ can be taken out of the double sum).
%%\end{enumerate}
%%\end{remark}

\begin{remark}\label{rem:convergence} By Corollary~\ref{cor:embedbf} we see that the convergence results remain valid whenever $\lambda\in \as$ (with $p^+,q^+<\infty$ in the case $a=f$). Moreover, a corresponding result also holds when we replace molecules by atoms since the latter are essentially special cases of the former up to multiplicative constants. Indeed, given $(K,L,d)$-atoms $a_{jm}$, for any $M>0$, $\big(1+\tfrac{d\sqrt{n}}{2}\big)^{-M}a_{jm}$ are $(K,L,M)$-molecules of a corresponding type (cf. Remark~\ref{atom-is-molecule} above). As long as we choose $M$ satisfying the second assumption in \eqref{conv-conditions} and take $\lambda\in \as$, we guarantee the convergence of
$$\sum_{j=0}^\infty \sum_{m\in\Zn} \lambda_{jm}\, \big(1+\tfrac{d\sqrt{n}}{2}\big)^{-M}a_{jm} \ \ \ \ \ \ \text{and} \ \ \ \ \ \sum_{(j,m)\in \Nz\times \Zn} \lambda_{jm}\, \big(1+\tfrac{d\sqrt{n}}{2}\big)^{-M}a_{jm}$$
in $\cS'$, where clearly the extra constant $\big(1+\tfrac{d\sqrt{n}}{2}\big)^{-M}$ can be taken out of the sums.
\end{remark}

%\begin{remark}\label{rem:convergence}\hspace{100mm}
%\begin{enumerate}
%\item[(ii)] The previous proposition holds also with $[aa]_{jm}$ replaced by $(K,L,d)$-atoms $a_{jm}$ when $j\in\N$ and by $(K,0,d)$-atoms $a_{0m}$ when $j=0$, under the same hypotheses except those on $M$ (which are not required then). Indeed, as observed in Remark~\ref{atom-is-molecule}, given such atoms, for any $M>0$, $\big(1+\tfrac{d\sqrt{n}}{2}\big)^{-M}a_{jm}$ are $(K,L,M)$- and $(K,0,M)$-molecules, respectively. As long as we choose $M$ satisfying the second assumption in \eqref{conv-conditions} and take $\lambda\in b^\w_{\px,\infty}$, we guarantee the convergence (in $\cS'$) of
%    $$\sum_{j=0}^\infty \sum_{m\in\Zn} \lambda_{jm}\, \big(1+\tfrac{d\sqrt{n}}{2}\big)^{-M}a_{jm}$$
%    (where clearly the constant $\big(1+\tfrac{d\sqrt{n}}{2}\big)^{-M}$ can be taken out of the double sum).
%% \item[(iv)] Under the assumptions of the proposition above, we can also show that
%%    $$\sum_{(j,m)\in\Nz\times\Zn} \lambda_{jm}\, [aa]_{jm}$$
%%    converges in $\cS'$ (to the same distribution to which the series in \eqref{conv} converges).
%\end{enumerate}
%\end{remark}

In the sequel we present atomic and molecular decompositions statements for the general spaces $\A$. Our statements introduce several improvements to the usual way of presenting corresponding results in the related literature when dealing with different particular cases. By these reasons, we give detailed formulations below.

Coming back to Proposition~\ref{pro:conv}, one can show that if $\lambda=(\lambda_{jm})$ belongs to suitable sequence spaces and one makes appropriate stronger assumptions on $K$, $L$ and $M$, then the distribution represented by \eqref{conv} belongs indeed to a certain corresponding function space and its quasi-norm is controlled from above by the corresponding quasi-norm of $\lambda$:

\begin{theorem}\label{thm:decomp1}
Let $\w\in\W$, $p,q\in\PPlog$, $K,L\in\Nz$ and $M>0$ be such that
\begin{equation}\label{conv-conditions1}
K>\alpha_2\,, \ \ \ L>\sigma_{p^-}-\alpha_1 +c_{\log}(1/q)\,, \ \ \ M>L+2n+2\alpha+\max\{1,2\,c_{\log}(1/p)\}\sigma_{p^-}+c_{\log}(1/q)
\end{equation}
when $a=b$ and $A=B$, or
\begin{equation}\label{conv-conditions2}
K>\alpha_2\,, \ \ \ L>\sigma_{p^-,q^-}-\alpha_1\,, \ \ \ M>L+2n+2\alpha+\max\{1,2\,c_{\log}(1/p)\}\,\sigma_{p^-,q^-},
\end{equation}
when $a=f$ and $A=F$ (with the additional restriction $p^+,q^+<\infty$ in this case). Let also $[aa]_{0m}$, $m\in\Zn$, be $(K,0,M)$-molecules and, for $j\in\N$ and $m\in\Zn$, $[aa]_{jm}$ be $(K,L,M)$-molecules. If $\lambda=(\lambda_{jm})\in \as$ then
\begin{equation}\label{eq:conv}
g:=\sum_{j=0}^\infty \sum_{m\in\Zn} \lambda_{jm}\, [aa]_{jm} \in \A.
\end{equation}
Moreover, there exists $c>0$ such that
\begin{equation}\label{aux4}
\|g\,| \A\| \le c\, \|\lambda\,| \as\|
\end{equation}
for all such molecules and all such $\lambda$.
\end{theorem}

Comments to the proof, including auxiliary results, are given in Subsection~\ref{sec:proof-decomp1} below.

%\begin{enumerate}
%\item[(i)]  If $\lambda=(\lambda_{jm})\in \bs$ and
%\begin{equation*}
%K>\alpha_2\,, \ \ \ L>\sigma_{p^-}-\alpha_1 +c_{\log}(1/q)\,, \ \ \ \ M>L+2n+2\alpha+\sigma_{p^-}\max\{1,2\,c_{\log}(1/p)\}+c_{\log}(1/q)
%\end{equation*}
%then
%\begin{equation*}
%g:=\sum_{j=0}^\infty \sum_{m\in\Zn} \lambda_{jm}\, [aa]_{jm} \in \B.
%\end{equation*}
%Moreover, there exists $c>0$ such that
%$$
%\|g\,| \B\| \le c\, \|\lambda\,| \bs\|
%$$
%for all such molecules and all such $\lambda$.
%
%\item[(ii)]  If $\lambda=(\lambda_{jm})\in \fs$ and
%\begin{equation*}
%K>\alpha_2\, \ \ \ L>\sigma_{p^-q^-}-\alpha_1 \,, \ \ \ M>L+2n+2\alpha+\sigma_{p^-q^-}\max\{1,2\,c_{\log}(1/p)\}
%\end{equation*}
%then
%\begin{equation*}
%g:=\sum_{j=0}^\infty \sum_{m\in\Zn} \lambda_{jm}\, [aa]_{jm} \in \F.
%\end{equation*}
%Moreover, there exists $c>0$ such that
%$$
%\|g\,| \B\| \le c\, \|\lambda\,| \fs\|
%$$
%for all such molecules and all such $\lambda$.
%
%\end{enumerate}

\begin{remark}\label{rem:atoms-also}
 The reader may note that the sufficient conditions on the size of $M$ in \eqref{conv-conditions1} and \eqref{conv-conditions2} slightly differ from the corresponding assumptions given in \cite[Remark~3.14]{Kem10}. It seems (private communication from the author) the latter contains some misprints in the calculations mentioned there. Arguing as in Remark~\ref{rem:convergence}, we can see that a similar result to the previous theorem holds with atoms instead of molecules, under the same assumptions (except those on $M$, since it plays no role then). Note that we recover standard assumptions when we deal with the classical spaces $A^s_{p,q}$ with constant exponents, since in that situation we have $\alpha=0$, $\alpha_1=\alpha_2=s$ and $c_{\log}(1/p)=c_{\log}(1/q)=0$.
\end{remark}

The next result shows that compactly supported functions with continuous and bounded derivatives up to a given order belong to the spaces $\A$ if that order is large enough.

\begin{corollary}
Let $K\in\Nz$ and $\psi$ be a $C^K$-function on $\Rn$ with compact support. Then $\psi\in\A$ for any $\w\in\W$ with $\alpha_2<K$ and $p,q\in\PPlog$ (with $p^+,q^+<\infty$ also in the case $A=F$).
\end{corollary}

\begin{proof}
Let $\displaystyle \bar{M}:=\max_{|\gamma| \le K} \sup_{x\in\Rn} \left|D^\gamma \psi(x)\right|$. Then $\bar{M}$ is finite by the hypotheses on $\psi$. There is nothing to prove if $\psi \equiv 0$, therefore we assume $\psi \not\equiv 0$, in which case $\bar{M}>0$. Since
$$ \sup_{x\in\Rn} \left|D^\gamma \frac{\psi}{\bar{M}}(x)\right| \le 1=2^{|\gamma|\,0}\, , \ \ \ \ 0\le |\gamma|\le K, \ \ \ \text{and} \ \ \ \ \supp \frac{\psi}{\bar{M}} \subset d\,Q_{00}$$
for $d>1$ chosen large enough, we see that $\tfrac{\psi}{\bar{M}}$ is a $(K,0,d)$-atom $a_{00}$. Consider the remaining $a_{j m}$, $j\in\Nz$, $m\in\Zn$, identically zero, so that they are $(K,L,d)$-atoms whatever the $L\in\Nz$ we choose. Define the sequence $\lambda=(\lambda_{j m})$ by taking  $\lambda_{j m}=\bar{M}$ for $j=0$ and $m=0$, and $\lambda_{j m}=0$ otherwise. It is not hard to see that $\lambda\in \as$. Indeed, we have, with $h_j= \sum_{m\in\Zn} \lambda_{j m} \,w_j(2^{-j}m) \chi_{j m}$,
$$\|\lambda\,|\, \bs\| = \| (h_j)_j \,|\, \ell_\qx(L_\px) \|= \bar{M}\, w_0(0)\,\|\chi_{00}\,|L_\px\| =  \bar{M}\, w_0(0)< \infty,$$
since the sequence $(h_j)_j$ has only one non-zero entry, namely $h_0=\bar{M}\, w_0(0)\,\chi_{00}$. The case $a=f$ is similar. Using the fact that $K>\alpha_2$, from Theorem~\ref{thm:decomp1} and Remark~\ref{rem:atoms-also} we conclude that
$$\psi = \sum_{j=0}^\infty \sum_{m\in\Zn} \lambda_{j m} \, a_{j m} \, \in \A\,,$$
which proves the claim.
\end{proof}

In particular, the previous corollary shows that $(K,L,d)$-atoms belong to the space $\A$ for any $p,q\in\PPlog$ (with $p^+,q^+<\infty$ in the case $A=F$) and $\w\in\W$ such that $\alpha_2<K$. Arguing as in the last part of the preceding proof, it is clear that the molecules in Theorem~\ref{thm:decomp1} belong also themselves to the $\A$ considered there.

The next theorem is a kind of reciprocal to Theorem~\ref{thm:decomp1}, now for atoms.

\begin{theorem}\label{thm:decomp2}
Let $K,L\in\Nz$ and $d>1$.
\begin{enumerate}
\item[(a)] For any $f\in\cS'$, there exist $\lambda(f):=(\lambda_{jm}) \subset \mathbb{C}$ and $(K,0,d)$-atoms $a_{0m}\in\cS$, $m\in\Zn$, and $(K,L,d)$-atoms $a_{jm}\in\cS$, $j\in\N$, $m\in\Zn$, such that
\begin{equation}\label{eq:fdecomp}
f=\sum_{j=0}^\infty \sum_{m\in\Zn} \lambda_{jm}\, a_{jm}
\end{equation}
where the inner sum is taken pointwisely and the outer sum converges in $\cS'$.

\item[(b)] If, moreover, $f$ belongs to some space $\A$, for some $\w\in\W$ and $p,q\in\PPlog$ (with the additional restriction $p^+,q^+<\infty$ in the case $A=F$), then both sums above converge in $\cS'$ and
$$
  \|\lambda(f)\,| \as\| \lesssim \|f\,| \A\|.
$$
\end{enumerate}
\end{theorem}

\begin{remark}\label{rem:3rem} \hspace{10cm}
\begin{enumerate}
%\item[(i)] Although we do not write down the proof, we would like to note that the inner sum in \eqref{eq:fdecomp} can also be considered in the pointwise sense with the same outcome as the convergence in $\cS'$.

\item[(i)] The choice of $K, L$ and $d$ above is completely independent of the parameters $p, q$ and $\w$. We are claiming that any $f\in\cS'$ can be represented by atoms as in \eqref{eq:fdecomp} for whatever $K,L\in\Nz$ and $d>1$ we choose. In particular, we can take $L=0$ and consequently no moment conditions are required for the atoms. As far as we can check, this interesting fact is never mentioned in the literature. It is also worth pointing out that the atomic representation \eqref{eq:fdecomp} we built is the same for each given $f\in\cS'$, independently of the space $\A$ where $f$ might belong to.

\item[(ii)] By comparison with \eqref{aux4} (and noting that a similar estimate holds for atoms, as mentioned above), we see that when all the hypotheses are met, our choice of atoms and coefficients in the previous theorem indeed satisfies
    $$ \|\lambda\,| \as\| \approx \|f\,| \A\|$$
    and we then say we have an \emph{optimal atomic decomposition} in $\A$ for $f$.
\item[(iii)] Even for the cases covered by former results in \cite[Theorem~3.12, Corollary~5.6]{Kem10}, and even when restricting to constant exponents (see, e.g., \cite[Step~4 on pp. 80--81]{Tri97}), Theorem~\ref{thm:decomp2} above contains more information and better assumptions. In particular, the convergence of the outer sum follows directly from the convenient reproducing formula used in the proof given in Subsection~\ref{sec:proof-decomp2}, and does not require the consideration of big values for $K$ nor $L$.
\end{enumerate}
\end{remark}

Taking into account all the comments made in the previous remark, we find it convenient to write down a complete proof for Theorem~\ref{thm:decomp2} (cf. Subsection~\ref{sec:proof-decomp2}).

\medskip

It is easy to see that the previous theorem can also be read with $(K,L,M)$-molecules (for given $K,L\in\Nz$ and $M>0$) instead of the corresponding atoms. Indeed,  from \eqref{eq:fdecomp} we have
    $$ f=\sum_{j=0}^\infty \sum_{m\in\Zn} \lambda_{jm}\big(1+d\tfrac{\sqrt{n}}{2}\big)^{M}\, \big(1+d\tfrac{\sqrt{n}}{2}\big)^{-M}a_{jm} $$
    where $\big(1+d\tfrac{\sqrt{n}}{2}\big)^{-M}a_{jm}$ are $(K,0,M)$-molecules if $j=0$ and $(K,L,M)$-molecules if $j\in\N$. Moreover, one has
    $$   \big\|\big(\lambda_{jm}\big(1+d\tfrac{\sqrt{n}}{2}\big)^{M}\big)_{j,m}\,| \as\big\| \lesssim \|f\,| \A\| $$
    with the implicit constant still independent of $f$.

Combining the previous results and remarks we get the following representations.

\begin{corollary}[atomic decomposition]
Let $\w\in\W$ and $p,q\in\PPlog$ (with $p^+,q^+<\infty$ in the $F$-case). Let $K,L\in\Nz$ and $d>1$ be such that
\begin{equation*}
K>\alpha_2\ \ \ \text{and} \ \ \ L>\sigma_{p^-}-\alpha_1 +c_{\log}(1/q)
\end{equation*}
in the $B$-case, or
\begin{equation*}
K>\alpha_2 \ \ \ \text{and} \ \ \ L>\sigma_{p^-,q^-}-\alpha_1
\end{equation*}
in the $F$-case. Then $f\in\cS'$ belongs to $\A$ if and only if there exist $\lambda=(\lambda_{jm})\in\as$ and $(K,0,d)$-atoms $a_{0m}$, $m\in\Zn$, and $(K,L,d)$-atoms $a_{jm}$, $j\in\N$, $m\in\Zn$, such that
\begin{equation}\label{eq:fdecomp-atom}
f=\sum_{j=0}^\infty \sum_{m\in\Zn} \lambda_{jm}\, a_{jm} \ \ \ \ \text{(convergence in $\cS'$)}.
\end{equation}
Moreover,
$$
\inf \, \| \lambda \,|\, \as \|
$$
defines a quasi-norm in $\A$ which is equivalent to $\| \cdot \,| \A\|$, where the infimum runs over all $\lambda\in\as$ that can be used in \eqref{eq:fdecomp-atom} (for fixed $f\in\A$) for all possible atoms with the properties described above.
\end{corollary}

\begin{corollary}[molecular decomposition]
Let $\w\in\W$ and $p,q\in\PPlog$ (with $p^+,q^+<\infty$ in the $F$-case). Let $K,L\in\Nz$ and $M>0$ satisfying \eqref{conv-conditions1}
in the $B$-case, or \eqref{conv-conditions2} in the $F$-case. Then $f\in\cS'$ belongs to $\A$ if and only if there exist $\lambda\in\as$ and $(K,0,M)$-molecules $[aa]_{0m}$, $m\in\Zn$, and $(K,L,M)$-molecules $[aa]_{jm}$, $j\in\N$, $m\in\Zn$, such that
\begin{equation}\label{eq:fdecomp-mol}
f=\sum_{j=0}^\infty \sum_{m\in\Zn} \lambda_{jm}\, [aa]_{jm} \ \ \ \ \text{(convergence in $\cS'$)}.
\end{equation}
Moreover,
$$
\inf \, \| \lambda \,|\, \as \|
$$
defines a quasi-norm in $\A$ which is equivalent to $\| \cdot \,| \A\|$, where the infimum runs over all $\lambda\in\as$ that can be used in \eqref{eq:fdecomp-mol} (for fixed $f\in\A$) for all possible molecules of the type described above.
\end{corollary}

\begin{remark}
We note that atomic representations for the particular spaces $B^{s(\cdot)}_{\px,\qx}$ were derived in \cite{Drihem12}. Furthermore, during the very final part of the preparation of this article, we became aware of the quite recent papers \cite{YZYuan15a,YZYuan15b} where the corresponding authors have obtained smooth molecular and atomic characterizations for generalized spaces of Triebel-Lizorkin and Besov type denoted by $F_{\px,\qx}^{s(\cdot), \phi}$ and $B_{\px,\qx}^{s(\cdot), \phi}$, respectively. However, although quite general, such scales do not include all the cases covered by the scales we are dealing with in our paper.
\end{remark}

\section{On the type of convergence and applications}\label{sec:converge-in-space}

This section aims to complement the results given in the previous one. In this part we discuss the type of convergence involving the atomic and molecular decompositions above. The main result, Theorem~\ref{thm:conv-space}, shows that the convergence implied in Theorem~\ref{thm:decomp1} holds not only in $\cS'$ but also in the spaces $\A$ themselves when the integrability parameters are bounded. As we can see below, this is an important fact having relevant implications and typically it is not mentioned in the standard literature even for classical spaces with constant exponents. We note that this problem was first explicitly studied for the classical Besov and Triebel-Lizorkin spaces in \cite{Cae11}.

We formulate the theorem below in terms of molecules, but clearly a corresponding result in terms of atoms holds as well.

\begin{theorem}\label{thm:conv-space}
Under the hypotheses of Theorem~\ref{thm:decomp1}, if $q^+<\infty$ then the convergence of the outer sum in \eqref{eq:conv} holds also in the spaces $\A$. Additionally, if also $p^+<\infty$ then the overall convergence in \eqref{eq:conv} to $g$ holds in $\A$. In fact, if both exponents $p$ and $q$ are bounded we even have
\begin{equation} \label{aux5}
g=\sum_{(j,m)\in\Nz\times\Zn} \lambda_{j m} \, [aa]_{j m}
\end{equation}
with summability in $\A$.
\end{theorem}

\begin{proof}
We give the proof for the $B$ space only. The $F$ counterpart can be done using similar arguments (with the necessary modifications).

\emph{Step 1}: We assume that $q^+<\infty$ and prove the convergence of the outer sum in \eqref{eq:conv} in $\B$. For each $T\in\N$, let $\lambda_T$ be the sequence formed from $\lambda$ by replacing the terms where $j\in\{0,1,\ldots,T\}$ by zero, and $\lambda^T$ be the sequence formed also from $\lambda$ but by replacing the terms where $j>T$ by zero. Then we still have $\lambda_T,\lambda^T \in \bs$. By Theorem~\ref{thm:decomp1} we have
$$
\B \ni g^T := \sum_{j=0}^T \sum_{m\in\Zn} \lambda_{j m}\, [aa]_{j m}
$$
and
$$
\B \ni g_T := \sum_{j=T+1}^\infty \sum_{m\in\Zn} \lambda_{j m}\, [aa]_{j m}
$$
(with convergence in $\cS'$). Moreover, the same theorem yields
\begin{equation*}
\|g-g^T\,|\B\| = \|g_T\,|\B\| \lesssim \|\lambda_T\,|\bs\|.
\end{equation*}
Since $q^+<\infty$ and $\lambda\in\bs$, from Lemma~\ref{lem:lpqconverge} we conclude that the quasi-norm of the right-hand side tends to zero as $T\to \infty$. We have then shown that
\begin{equation*}
g= \sum_{j=0}^\infty \sum_{m\in\Zn} \lambda_{j m}\, [aa]_{j m}
\end{equation*}
with convergence of the outer sum in $\B$ (and where the inner sum can either be interpreted pointwisely or in $\cS'$).

\emph{Step 2}: In this step we assume that both $p$ and $q$ are bounded and then prove the summability of \eqref{aux5} in $\B$. Let $\varepsilon>0$ be arbitrary small. We want to show that there exists a finite part $P_0$ of $\Nz\times \Zn$ such that
\begin{equation*}
\left\| g- \sum_{(j ,m)\in P} \lambda_{j m}\,[aa]_{j m}\;\big|\, \B\right\| \le \varepsilon,
\end{equation*}
for all finite parts $P$ of $\Nz\times \Zn$ containing $P_0$. Consider
$$P_0:= \{(j,m)\in\Nz\times \Zn:\, j\le j_0,\,\;|m|\le m_0\}$$
with $j_0\in\Nz$ and $m_0\in\Zn$ to be chosen (depending on $\varepsilon$). For finite $P \subset \Nz\times \Zn$, let $\lambda^{\overline{P}}$ be the sequence given by $\lambda_{j m}^{\overline{P}} = \lambda_{jm}$ if $(j,m)\in P$ and $\lambda_{j m}^{\overline{P}} = 0$ otherwise. We can write
%\[
%\sum_{(j,m)\in P} \lambda_{j m} \, [aa]_{j m} = \sum_{j=0}^{j_1} \sum_{m\in\Zn\atop |m|\le m_1} \lambda_{j m}^{\overline{P}} \, [aa]_{jm}
%\]
%for some $j_1,m_1 \in \Nz$. However, since all the terms $\lambda_{j m}^{\overline{P}}$ vanish for $j>j_1$ or $|m|>m_1$, then also
%\[
%\sum_{(j,m)\in P} \lambda_{j m} \, [aa]_{j m} = \sum_{j=0}^{\infty} \sum_{m\in\Zn} \lambda_{j m}^{\overline{P}} \, [aa]_{jm} \ \ \ \ \ \text{with convergence in $\cS'$}.
%\]
%Consequently,
%Then we have
\begin{eqnarray*}
\left\| g- \sum_{(j,m)\in P} \lambda_{j m}\,[aa]_{j m}\;\big|\, \B\right\| & = & \left\| \sum_{(j,m)\notin P} \lambda_{jm}\, [aa]_{jm}   \;\big|\, \B\right\|\\
& = & \left\| \sum_{j=0}^\infty \sum_{m\in\Zn} \lambda_{j m}^P\,[aa]_{j m}\;\big|\, \B\right\| \\
& \le & c\,\left\| \lambda^P\;|\, \bs\right\|,
\end{eqnarray*}
naturally taking $\lambda_{j m}^{P} = \lambda_{jm}$ if $(j,m)\notin P$ and $\lambda_{j m}^{P} = 0$ otherwise, with the inequality following from Theorem~\ref{thm:decomp1}. Note that the sums above converge in $\cS'$ (cf. Proposition~\ref{pro:conv}).
By homogeneity, showing that the last quasi-norm above is at most $\varepsilon$ it is equivalent to show that $\left\|\tfrac{c}{\varepsilon} \lambda^P\;|\, \bs\right\|\le 1$. On the other hand, by the unit ball property, this is equivalent to prove that the corresponding semimodular is less than or equal to one.  We may assume $\varepsilon \le c$ in the sequel. We have
\begin{eqnarray}
\varrho_{\ell_\qx(L_\px)} \!\left( \!\tfrac{c}{\varepsilon} \!\left( \sum_{m\in\Zn} \!\lambda_{j m}^P\, w_j(2^{-j}m)\, \chi_{j m}\right)_j \!\right)\! & = & \! \big(\tfrac{\varepsilon}{c}\big)^{-q^+} \!\sum_{j =0}^\infty \left\|  \left|\sum_{m\in\Zn} \lambda_{j m}^P\, w_j(2^{-j}m)\, \chi_{j m}\right|^{\qx} \!\big| L_{\frac{\px}{\qx}}\right\| \nonumber \\
& \le & \!\big(\tfrac{\varepsilon}{c}\big)^{-q^+}\! \left[ \sum_{j =0}^{j_0} \left\|  \cdots \big| L_{\frac{\px}{\qx}}\right\| + \sum_{j =j_0+1}^{\infty} \left\|  \cdots \big| L_{\frac{\px}{\qx}}\right\| \right].\label{firstsum}
\end{eqnarray}
As in the first step we can see that the second sum goes to zero when $j_0 \to \infty$. Hence we can find $j_0\in\N$ such that this sum is smaller than $\tfrac{1}{2}\big(\tfrac{\varepsilon}{c}\big)^{q^+}$. As regards the first sum, we have
\begin{eqnarray*}
\lefteqn{\sum_{j =0}^{j_0} \Big\|  \Big|\sum_{m\in\Zn} \lambda_{j m}^P\, w_j(2^{-j}m)\, \chi_{j m}\Big|^{\qx} \,\big| L_{\frac{\px}{\qx}}\Big\|  \le \sum_{j =0}^{j_0} \left\|  \sum_{|m|>m_0} |h_{j m}|^{\qx} \, \big| L_{\frac{\px}{\qx}}\right\| } \\
& \!\!\!\!\le & \sum_{j =0}^{j_0} \max \left\{ \varrho_{\frac{\px}{\qx}} \left( \sum_{|m|>m_0} |h_{j m}|^{\qx} \right)^{\frac{q^-}{p^+}}, \varrho_{\frac{\px}{\qx}} \left( \sum_{|m|>m_0} |h_{j m}|^{\qx} \right)^{\frac{q^+}{p^-}} \right\}
\end{eqnarray*}
using the notation $ h_{j m}:= \lambda_{j m}\, w_j(2^{-j}m)\, \chi_{j m}$ and estimate \eqref{Lp-norm-mod} in the last inequality (note that $\tfrac{p}{q}$ is bounded). Observe that
\begin{equation*}
\varrho_{\frac{\px}{\qx}} \left( \sum_{|m|>m_0} |h_{j m}|^{\qx} \right) = \int_{\Rn} \Big[\sum_{|m|>m_0} |h_{jm}(x)|\Big]^{p(x)}dx.
\end{equation*}
Let
$$
g_{j,J}(x):= \sum_{i=0}^{J} \sum_{i\le |m| < i+1} |h_{j m}(x)|  \,, \ \ J\in\Nz, \ \ \ \text{and} \ \ \ g_j(x):= \sum_{m\in\Zn} |h_{j m}(x)|.
$$
Then $0\le g_{j,J} \le g_j$ and $g_{j,J} \to g_j$ pointwisely as $J\to \infty$. Moreover, the hypothesis $\lambda\in\bs$ implies, in particular, that $g_j\in L_{\px}$. Since $p^+<\infty$ we can apply the Lebesgue dominated convergence theorem (cf. \cite[Theorem~2.62]{C-UF13}) and conclude that $g_{j,J} \to g_j$ in $L_\px$ as $J\to \infty$. This implies
$$
\int_{\Rn} |g_{j,J}(x)-g_{j}(x)|^{p(x)}dx \rightarrow 0 \ \ \ \text{as} \ \ \ J \to \infty,
$$
so that
$$
\int_{\Rn} \Big[ \sum_{|m|>J} |h_{j m}(x)|\Big]^{p(x)}dx \rightarrow 0 \ \ \ \text{as} \ \ \ J \to \infty.
$$
Therefore, for each $j\in\{0,1,\ldots,j_0\}$, there exists $m(j)\in\N$ such that
$$
\max \left\{ \varrho_{\frac{\px}{\qx}} \left( \sum_{|m|>m(j)} |h_{j m}|^{\qx} \right)^{\frac{q^-}{p^+}}, \varrho_{\frac{\px}{\qx}} \left( \sum_{|m|>m(j)} |h_{j m}|^{\qx} \right)^{\frac{q^+}{p^-}} \right\} < \tfrac{1}{2(j_0+1)}\big(\tfrac{\varepsilon}{c}\big)^{q^+}.
$$
Taking $m_0:=\max\{m(0),m(1),\ldots,m(j_0)\}$, we conclude that the first sum in \eqref{firstsum} is dominated by $\sum_{j=0}^{j_0} \tfrac{1}{2(j_0+1)}\big(\tfrac{\varepsilon}{c}\big)^{q^+} = \tfrac{1}{2}\big(\tfrac{\varepsilon}{c}\big)^{q^+}$. Putting everything together, we have proved our claim.
\end{proof}

%\begin{remark}For the proof of the convergence of the outer sum in the space $\B$ the hypotheses $p^+<\infty$ is not needed.
%As in previous results, the theorem above also holds with corresponding atoms instead of molecules, in the sense of Remark~\ref{rem:atoms-also}.
%\end{remark}

An interesting application of the previous theorem is the denseness of nice classes in the general $2$-microlocal spaces with variable integrability. We can derive simple proofs of such properties by taking advantage of the convergence in the spaces themselves. We explore this in the next corollary and remark.

%Let $C_b^\infty$ denote the class of all complex valued functions on $\Rn$ having bounded continuous derivatives of all orders.

%\begin{corollary}\label{cor:dense}
%If $\w\in\W$ and $p,q\in\PPlog$ with $q^+<\infty$, then $C_b^\infty \cap \A$ is dense in $\A$ (with $p^+<\infty$ in the case $A=F$). If, in addition, $\max\{p^+,q^+\}<\infty$ then $\cS$ is also dense in $\A$.
%\end{corollary}

\begin{corollary}\label{cor:dense}
If $\w\in\W$ and $p,q\in\PPlog$ with $\max\{p^+,q^+\}<\infty$, then $\cS$ is dense in $\A$.
\end{corollary}

\begin{proof}
For fixed $p,q$ and $\w$, consider $f\in\A$ and apply Theorem~\ref{thm:decomp2} to decompose $f$ into a linear combination of Schwartz atoms as
\begin{equation*}
f=\sum_{\nu=0}^\infty \sum_{m\in\Zn} \lambda_{\nu m}\, a_{\nu m} \ \ \ \ \text{(convergence in $\cS'$)}.
\end{equation*}
So here $a_{0m}\in\cS$, $m\in\Zn$, and $a_{\nu m}\in\cS$, $\nu\in\N$, $m\in\Zn$, are $(K,0,d)$-atoms and $(K,L,d)$-atoms, respectively, and $\lambda=(\lambda_{\nu m})_{\nu,m} \in \as$. Moreover, $K,L\in\Nz$ and $d>1$ are at our disposal. If we have been careful enough to consider $K$ and $L$ according to the conditions \eqref{conv-conditions1} in the case $A=B$, or according to \eqref{conv-conditions2} in the case $A=F$, then from Theorem~\ref{thm:conv-space} (and the comments preceding it) we see that not only the convergence of the double sum above holds in $\A$, but we even have
\[
f=\sum_{(\nu,m)\in\Nz\times\Zn} \lambda_{\nu m} \, a_{\nu m} \ \ \ \ \text{ with summability in $\A$.}
\]
Since $\Nz\times \Zn$ is infinitely countable, we can write
\[
f=\sum_{j=1}^\infty \lambda_{j} \, a_{j} \ \ \ \ \text{(convergence in $\A$)}
\]
after reindexing $\lambda_{\nu m}$ and $a_{\nu m}$. This means that
\[
f=\lim_{J\to\infty}\sum_{j=1}^J \lambda_{j} \, a_{j} \ \ \ \text{in $\A$}.
\]
Since the partial sums belong to $\cS$ the proof is complete.
\end{proof}

We note that the denseness of $\cS$ in $\A$ (for bounded $p,q$, and constant $q$ in the $B$ case) was recently referred in \cite{GonMN14} and \cite{MouNS13}, with the suggestion that the proof be done following classical arguments as in \cite[Theorem~2.3.3]{Tri83}.

\begin{remark}
We can also say, in the case $q^+<\infty$ and $p^+=\infty$ (then we are dealing with $B$ spaces), that each $f\in\B$ can be approximated in $\B$ by the sequence $(f_\nu)_\nu$ of quite regular functions given by $f_\nu=\sum_{m\in\Zn} \lambda_{\nu m}\,a_{\nu m}$, with atoms and coefficients as in the proof of Corollary~\ref{cor:dense}, for suitable large $K,L\in\Nz$. It is easy to see that such functions $f_\nu$ have bounded and uniformly continuous derivatives of all orders (see the specific construction given in Subsection~\ref{sec:proof-decomp2}).
\end{remark}

Another useful application of Theorem~\ref{thm:conv-space}, again combined with Theorem~\ref{thm:decomp2}, is the following result concerning the possibility of the atomic representations to converge pointwisely:

\begin{corollary}
Let $\w\in\mathcal{W}^{\alpha}_{\alpha_1,\alpha_2}$, $p,q\in\PPlog$, $K,L\in\Nz$ and $d>1$ be such that $K>\max\{\alpha_2,0\}$  and 
$$
L>\sigma_{p^-}-\alpha_1 + c_{\log}(1/q) \ \ \text{in the $B$ case} \ \ \ \text{or} \ \ \ L>\sigma_{p^-,q^-}-\alpha_1 \ \ \text{in the $F$ case}.
$$
Then, given any $f\in\cS$, the atoms $a_{jm}$ and coefficients $\lambda_{jm}$ given in Theorem~\ref{thm:decomp2} for such $f$ give an optimal representation
$$
f=\sum_{(j,m)\in\Nz\times\Zn} \lambda_{j m} \, a_{j m}
$$
in $\A$ with the convergence holding pointwisely (and even in a uniform way).
\end{corollary}

\begin{proof}
\emph{Step 1}: By Theorem~\ref{thm:conv-space} and the comment preceding it, any sum
\begin{equation}\label{aux8}
\sum_{(j,m)\in\Nz\times\Zn} \lambda_{j m} \, a_{j m}
\end{equation}
with $(K,0,d)$-atoms $a_{0m}$, $m\in\Zn$, $(K,L,d)$-atoms $a_{jm}$, $j\in\N$, $m\in\Zn$, and coefficients $(\lambda_{jm})\in b^s_{t,1}$, with $0<t<\infty$, $K>s$, $L>\sigma_t-s$, $d>1$, holds with convergence in $B^s_{t,1}$. By classical Sobolev embeddings for Besov spaces (see, e.g., \cite[Theorem~2.7.1(i)]{Tri83}), $B^s_{t,1} \hookrightarrow B^0_{\infty,1}$ if $s=\tfrac{n}{t}$, where $B^0_{\infty,1} \hookrightarrow \mathcal{C}$ (see \cite[Proposition~2.5.7]{Tri83}), the $\mathcal{C}$ standing for the space of all (complex-valued) bounded and uniformly continuous functions on $\Rn$, endowed with the sup-norm. Putting everything together, if $s=\tfrac{n}{t}$ then the convergence in \eqref{aux8} is not only pointwise but even uniform. Noticing that $\sigma_t=0$ if $t\ge 1$ and that we can take, in the framework above, $t$ finite as big as one wants and $s$ positive correspondingly close to $0$ (by the relation $s=\tfrac{n}{t}$), we see that given $K>0$, $L\ge 0$, $d>1$, it is always possible to choose $s$ and $t$ so that the requirements above are satisfied. Assume that such a choice has been made.

\emph{Step 2}: Since $\cS$ is contained in any Besov space, using Theorem~\ref{thm:decomp2} with $B^s_{t,1} $ instead of $\A$, we can say, given $f\in\cS$, that the construction in \eqref{eq:fdecomp} indeed satisfies $(\lambda_{jm})\in b^s_{t,1}$, and therefore the convergence implied there holds also in the sense of \eqref{aux8}, in particular in a pointwise and even uniform way. On the other hand, if one wants to think about $f$ as an element of $\A$, then the same coefficients $(\lambda_{jm})$ also belong to $\as$ and, in case we have been careful enough to have chosen previously $K>\alpha_2$, and $L>\sigma_{p^-}-\alpha_1 + c_{\log}(1/q)$ in the $B$ case or  $L>\sigma_{p^-,q^-}-\alpha_1$ in the $F$ case, we have, by Theorem~\ref{thm:decomp1}, that the obtained \eqref{eq:fdecomp} is indeed an optimal atomic representation of $f$ in $\A$ -- see also Remark~\ref{rem:3rem} (ii).
\end{proof}

%%%%%%%%%%%%%%%%%%%%%%%%%%%%%%%%%%%%%%%%%%%%%%
%%%%%%%%%%%%%%%%%%%%%%%%%%%%%%%%%%%%%%%%%%%%%%

\section{Applications to Sobolev type embeddings}\label{sec:Sobolev-embed}

We have already observed that the proof of Proposition~\ref{pro:conv} (given in Subsection~\ref{sec:proof-pro}) takes advantage of a Sobolev type embedding for the sequence spaces  $\bs$. Although the particular case $q=\infty$ is enough in that part, we take the opportunity to give a general result for arbitrary $q$. In this way, we will be able to obtain a Sobolev embedding for the $2$-microlocal Besov spaces $\B$ later on.

\begin{theorem}\label{thm:discrete-Sobolev-embed}
Let $\textbf{w}^\mathbf{1}\in\mathcal{W}^{\alpha}_{\alpha_1,\alpha_2}$, $p_0,p_1\in\PPlog$ and $\tfrac{1}{q}$ be
locally $\log$-Hölder continuous. If $p_0\le p_1$ and
$$
\frac{w^0_j(x)}{w^1_j(x)} = 2^{j\left(\frac{n}{p_0(x)}-\frac{n}{p_1(x)}\right)}, \ \ x\in\Rn, \ \ j\in\Nz,
$$
then
$$
b^{\mathbf{w}^\mathbf{0}}_{\pzx,\qx}  \hookrightarrow b^{\mathbf{w}^\mathbf{1}}_{\pumx,\qx}.
$$
\end{theorem}

Notice that the assumptions in the theorem imply that $\textbf{w}^\mathbf{0}$ is also an admissible weight sequence.

If $q$ is constant (situation already discussed in \cite[Proposition~3.9]{Kem10}), we can easily see that the embedding above follows from the Nikolskii type inequality
\begin{equation}\label{nikolskii}
\left\| \sum_{m\in\Zn} \lambda_{jm} \,w_j(\cdot)\,\chi_{jm} \,\big|\,L_{\pumx} \right\| \le c\,
\left\| \sum_{m\in\Zn} \lambda_{jm} \,w_j(\cdot)\,2^{j(\frac{n}{\pzx}-\frac{n}{\pumx})}\,\chi_{jm} \,\big|\,L_{\pzx} \right\|,
\end{equation}
where $c>0$ does not depend on $j\in\Nz$ nor on $\lambda=(\lambda_{jm})$.

The proof of \eqref{nikolskii} can be done by adapting the arguments from \cite[Lemma~6.3]{AlmH10}  (see also \cite[Lemma~2.9]{MouNS13}).

When $q$ is variable the situation is more complicated since an additional error term appears in the inequality. In the sequel we will see how to deal with
the general case. Note that we can assume here $q^-<\infty$ since otherwise it can be treated as the constant exponent $q=\infty$.

\begin{proposition}\label{pro:nikolskii}
Let $p_0,p_1\in\PPlog$ with $p_0\le p_1$, and $\tfrac{1}{q}$ be locally $\log$-Hölder continuous with $q^-<\infty$. Let also $(w_j)_j$ be a sequence satisfying \eqref{aux6} for some $\alpha \ge 0$. Then there exists $c_0\in(0,1]$ such that
\begin{eqnarray*}
\lefteqn{\inf\left\{\mu_j>0: \, \varrho_{\pumx}\left( \frac{c_0\sum_{m\in\Zn} \lambda_{jm} \,w_j\,\chi_{jm}}{\mu_j^{\frac1{\qx}}}\right)\le 1 \right\}}\\
& \!\!\!\!\! \le & \inf\left\{\mu_j>0: \, \varrho_{\pzx}\left( \frac{\sum_{m\in\Zn} \lambda_{jm} \,w_j\,2^{j(\frac{n}{\pzx}-\frac{n}{\pumx})}\,\chi_{jm}}{\mu_j^{\frac1{\qx}}}\right)\le 1 \right\} \, + \,2^{-j}\,,
\end{eqnarray*}
for all $j\in\Nz$ and all $(\lambda_{jm}) \subset \C$, provided the infimum on the right-hand side is at most one.
\end{proposition}

Using the unit ball property, the above generalized Nikolskii's inequality and the estimate from Lemma~\ref{lem:lqLp-norm-mod} combined with  scaling arguments, it is not hard to get the inequality
$$
\| \lambda \,|\, b^{\mathbf{w}^\mathbf{1}}_{\pumx,\qx} \| \le 3^{1/q^-}c_0^{-1} \, \| \lambda \,|\, b^{\mathbf{w}^\mathbf{0}}_{\pzx,\qx} \|
$$
($c_0$ is the constant from Proposition~\ref{pro:nikolskii}), from which the embedding in Theorem~\ref{thm:discrete-Sobolev-embed} follows.

For the proof of the above proposition we need the following  technical lemma. It is very easy to prove, so that we omit details.
\begin{lemma}
For any $R>0$ we have
\begin{equation}\label{aux7}
\chi_{jm}(x) \lesssim \left(\eta_{j,R} \ast \chi_{jm}\right)(x)
\end{equation}
with the implicit constant independent of $x\in\Rn$, $j\in\Nz$ and $m\in\Zn$.
\end{lemma}

%The next lemma is a generalization of \cite[Lemma~6.1]{DieHR09} to the $2$-microlocal setting. A proof can be found in \cite[p.1242]{MouNS13}.
%
%\begin{lemma}\label{lem:eta-change}
%Let $\alpha \ge 0$ and $\w=(w_j)_{j\in\Nz}$ be a positive weight sequence satisfying property (i) (i.e. the first property defining admissible weight sequences in Section~\ref{sec:microlocalspaces}). If $R \ge \alpha$, then
%\begin{equation}\label{chi-eta}
%w_j(x)\, \eta_{j,2R}(x-y) \lesssim w_j(y)\, \eta_{j,R}(x-y)
%\end{equation}
%with the implicit constant independent of $x,y\in\Rn$ and $j\in\Nz$.
%\end{lemma}
%
%In particular, the previous lemma allows us to treat the weights in
%many cases as if they did not depend on the space variable $x$, namely we can move them
%inside the convolution (for large $R$) as follows:
%\[
% w_j(x)\, (\eta_{j,2R} * f)(x) \lesssim \eta_{j,R}*(w_j f) (x).
%\]

\noindent\textit{Proof of Proposition~\ref{pro:nikolskii}}.$\;$ Technically the case $q^+=\infty$ is more complicated to deal with, but it can be handled by using similar arguments to the case $q^+<\infty$. We give details for the latter only. The proof below is based on ideas from \cite[Lemma~6.3]{AlmH10}.

If $q^+<\infty$ then we want to
show that
\begin{equation}\label{nikolskii1}
\left\| \Big|c_0\,\sum_{m\in\Zn} \lambda_{jm} \,w_j\,\chi_{jm}\Big|^{\qx} \,\big|\,L_{\frac{\pumx}{\qx}} \right\| \le
\left\| \Big|\sum_{m\in\Zn} \lambda_{jm} \,w_j\,2^{j(\frac{n}{\pzx}-\frac{n}{\pumx})}\,\chi_{jm}\Big|^{\qx} \,\big|\,L_{\frac{\pzx}{\qx}} \right\| + 2^{-j}
\end{equation}
when the quasi-norm of the right-hand side is less than or equal to one. Let $v_j(x):=w_j(x)\,2^{-j\frac{n}{p_1(x)}}$, $x\in\Rn$, $j\in\Nz$,
and let $\delta_j$ denote the right-hand side of \eqref{nikolskii1}. For any $r>0$, $j\in\Nz$, and almost every $x\in\Rn$, we have
\begin{eqnarray*}
\delta_j^{-\frac{r}{q(x)}} \, \Big| \sum_{m\in\Zn} \lambda_{jm} \,v_j(x)\,\chi_{jm}(x) \Big|^r & = & \delta_j^{-\frac{r}{q(x)}} \, \sum_{m\in\Zn} |\lambda_{jm}|^r \,v_j(x)^r\,\chi_{jm}(x)\,\chi_{jm}(x) \\
& \le & c_1\, \delta_j^{-\frac{r}{q(x)}} \, \sum_{m\in\Zn} |\lambda_{jm}|^r \,v_j(x)^r\,\chi_{jm}(x)\,\big(\eta_{j,R}\ast \chi_{jm}\big)(x)\\
& \le & c_1\,c_2 \int_{\Rn}\sum_{m\in\Zn} |\lambda_{jm}|^r \,\delta_j^{-\frac{r}{q(y)}}\,v_j(y)^r\,\eta_{j,R}(x-y) \chi_{jm}(y)\,dy
\end{eqnarray*}
where $R>0$ is arbitrary, $c_1>0$ is the constant from \eqref{aux7} and $c_2>0$ comes from moving the quantity $\delta_j^{-\frac{r}{q(x)}}\,v_j(x)^r$ inside the convolution. We show briefly how this can be done. Clearly, it is enough to prove that
$$
\delta_j^{-\frac{1}{q(x)}} \,v_j(x) \approx \delta_j^{-\frac{1}{q(y)}} \,v_j(y)\,, \ \ \ \ x,y\in Q_{jm},
$$
with the implicit constants independent of $x,y,j,m$ and $(\lambda_{jm})$. The problematic part here is the first factor in the product. However, this can be dealt with using the assumption on the quasi-norm on the right-hand side of \eqref{nikolskii1}, which implies $\delta_j \in[2^{-j}, 1+2^{-j}]$, so that $-1 \le \tfrac{\log_2 \delta_j}{j} \le 1$ for $j\in\N$, together with the local $\log$-Hölder continuity of $\tfrac{1}{q}$.

Taking $r\in(0,p_0^-)$ we can dominate the right-hand side of the previous inequality by
$$ 2c_1\,c_2\,c_3 \left\| \Big|\sum_{m\in\Zn} \lambda_{jm} \,\delta_j^{-\frac{1}{\qx}}\,v_j(\cdot)\, 2^{j\frac{n}{\pzx}}\,\chi_{jm}\Big|^r \,\big|\, L_{\frac{\pzx}{r}} \right\| $$
after applying Hölder's inequality with exponent $p_0/r > 1$, with
$$
\left\| 2^{-j\frac{nr}{\pzx}}\,\eta_{j,R}(x-\cdot) \, \big|\, L_{(\frac{\pzx}{r})'} \right\| \le c_3.
$$
Note that we have $c_3 \in (0,\infty)$, independent of $x$ and $j$, if we choose $R$ such that $R\big((\tfrac{\pzx}{r})^\prime\big)^->n$. To see this just calculate the  semimodular and use \eqref{Lp-norm-mod} to estimate the corresponding norm. Taking $c_0^{-r}:= 2c_1\,c_2\,c_3$ (which can always be assumed to be at least one and independent of $j,x$ and $(\lambda_{jm})$), from the definition of $\delta_j$ we get
$$
c_0\,\delta_j^{-\frac{1}{q(x)}} \, \Big| \sum_{m\in\Zn} \lambda_{jm} \,v_j(x)\,\chi_{jm}(x) \Big| \le \left\| \sum_{m\in\Zn} \lambda_{jm} \,\delta_j^{-\frac{1}{\qx}}\,v_j(\cdot)\, 2^{j\frac{n}{\pzx}}\,\chi_{jm} \,\big|\, L_{\pzx} \right\| \le 1.
$$
Using the fact that $c_0\in(0,1]$ and that $p_1(x) \ge p_0(x)$, we can conclude that
$$
\varrho_{\frac{\pumx}{\qx}} \left( \delta_j^{-1}\, \Big| c_0 \sum_{m\in\Zn} \lambda_{jm} \,w_j(\cdot)\,\chi_{jm} \Big|^{\qx}  \right) \le  \varrho_{\frac{\pzx}{\qx}} \left( \delta_j^{-1}\, \Big|\sum_{m\in\Zn} \lambda_{jm} v_j(\cdot)\, 2^{j\frac{n}{\pzx}}\,\chi_{jm} \Big|^{\qx}  \right)  \le 1
$$
where the last inequality follows from the unit ball property, observing that the corresponding quasi-norm is at most one by the definition of $\delta$ and $v_j$. By the same property, we have
$$
\left\| \delta_j^{-1}\, \Big| c_0 \sum_{m\in\Zn} \lambda_{jm} \,w_j(\cdot)\,\chi_{jm} \Big|^{\qx} \,\big|\, L_{\frac{\pumx}{\qx}} \right\| \le 1,
$$
from which the inequality \eqref{nikolskii1} follows by homogeneity.\hfill$\square$

\begin{corollary}\label{cor:Sobolev-embed}
Let $\textbf{w}^\mathbf{1}\in\mathcal{W}^{\alpha}_{\alpha_1,\alpha_2}$ and $p_0,p_1,q\in\PPlog$. If $p_0\le p_1$ and
$$
\frac{w^0_j(x)}{w^1_j(x)} = 2^{j\left(\frac{n}{p_0(x)}-\frac{n}{p_1(x)}\right)}, \ \ x\in\Rn, \ \ j\in\Nz,
$$
then
$$
B^{\mathbf{w}^\mathbf{0}}_{\pzx,\qx}  \hookrightarrow B^{\mathbf{w}^\mathbf{1}}_{\pumx,\qx}.
$$
\end{corollary}

\begin{proof}
The idea is to transfer the embedding given in Theorem~\ref{thm:discrete-Sobolev-embed} to the Besov spaces via atomic decompositions.
Consider $K,L\in\Nz$ such that $K>\alpha_2$ and $L>\sigma_{p_1^-}-\alpha_1 +c_{\log}(1/q)$. Let $d>1$. By Theorem~\ref{thm:decomp2}, for each $f\in B^{\mathbf{w}^\mathbf{0}}_{\pzx,\qx}$ we can write
\begin{equation*}\label{fdecomp}
f=\sum_{j=0}^\infty \sum_{m\in\Zn} \lambda_{jm}\, a_{jm} \ \ \ \ \text{(convergence in $\cS'$)}
\end{equation*}
with
\begin{equation}\label{aux1}
\| \lambda(f) \,|\, b^{\mathbf{w}^\mathbf{0}}_{\pzx,\qx}\| \le c \, \| f \,|\, B^{\mathbf{w}^\mathbf{0}}_{\pzx,\qx}\|,
\end{equation}
where $c>0$ is independent of $f$, and $a_{0m}$, $m\in\Zn$, are $(K,0,d)$-atoms and $a_{jm}$, $j\in\N$, $m\in\Zn$, are $(K,L,d)$-atoms. From Theorem~\ref{thm:discrete-Sobolev-embed} we have that
\begin{equation}\label{aux2}
\| \lambda(f) \,|\, b^{\mathbf{w}^\mathbf{1}}_{\pumx,\qx} \| \le c' \, \| \lambda(f) \,|\, b^{\mathbf{w}^\mathbf{0}}_{\pzx,\qx} \|
\end{equation}
with $c'>0$ not depending on $\lambda(f)$. Due to our choice of $K$ and $L$, Theorem~\ref{thm:decomp1} and Remark~\ref{rem:atoms-also} guarantee then the given $f$ belongs also to $B^{\mathbf{w}^\mathbf{1}}_{\pumx,\qx}$ and
\begin{equation}\label{aux3}
\| f \,|\, B^{\mathbf{w}^\mathbf{1}}_{\pumx,\qx}\| \le c''\, \| \lambda(f) \,|\, b^{\mathbf{w}^\mathbf{1}}_{\pumx,\qx} \|,
\end{equation}
with $c''>0$ independent of $f$. Combining \eqref{aux1}-\eqref{aux3} we get the required inequality.
\end{proof}

\begin{remark}
For constant $q$ (and even more when $p$ is also constant) there is a real advantage in proving Sobolev type embeddings for Besov spaces via atomic representations, since a corresponding discrete version is easy to obtain. This is not exactly the case when we deal with variable exponents $q$. In fact, the proof of Theorem~\ref{thm:discrete-Sobolev-embed} above already required a complicated result playing the role of classical Nikolskii's inequalities. An alternative approach to the embedding in Corollary~\ref{cor:Sobolev-embed} would be to show it directly from an adapted version of Proposition~\ref{pro:nikolskii} (cf. \cite[Lemma~6.3]{AlmH10} in the case of variable smoothness).
\end{remark}

The previous corollary can be seen as an extension of \cite[Theorem~6.4]{AlmH10} to the $2$-microlocal setting.

\begin{corollary}
Let $\textbf{w}^\mathbf{0}\in\W$, $\textbf{w}^\mathbf{1}\in\mathcal{W}^{\beta}_{\beta_1,\beta_2}$ and $p_0,p_1,q_0,q_1\in\PPlog$ with
$$
1\le \frac{w^0_j(x)}{w^1_j(x)} = 2^{j\big(\frac{n}{p_0(x)}-\frac{n}{p_1(x)}+\varepsilon(x)\big)}, \ \ \ \ x\in\Rn, \ \ j\in\Nz.
$$
If $\varepsilon^->0$ then
$$
B^{\mathbf{w}^\mathbf{0}}_{\pzx,\qzx}  \hookrightarrow B^{\mathbf{w}^\mathbf{1}}_{\pumx,\qumx}.
$$
\end{corollary}

\begin{proof}
Let $w_j^\varepsilon(x):= w^1_j(x)\,2^{j\varepsilon(x)}, \ \ j\in\Nz, \ \ x\in\Rn$. By Corollary~\ref{cor:Sobolev-embed} and Proposition~\ref{pro:embed1},
$$
B^{\mathbf{w}^\mathbf{0}}_{\pzx,\qzx} \hookrightarrow B^{\mathbf{w}^\mathbf{0}}_{\pzx,\infty} \hookrightarrow B^{\mathbf{w}^\mathbf{\varepsilon}}_{\pumx,\infty} \hookrightarrow B^{\mathbf{w}^\mathbf{1}}_{\pumx,\qumx}.
$$
In the last embedding we used the fact that $\big(\tfrac{w^1_j}{w^\varepsilon_j}\big)_j = (2^{-j\varepsilon(\cdot)})_j \in \ell_{q^\ast}(L_\infty)$ with $\tfrac{1}{q^\ast}=\big( \tfrac{1}{q^-_1}- \tfrac{1}{\infty}\big)_+ = \tfrac{1}{q^-_1}$, which is ensured by the hypotheses $\varepsilon^->0$.
\end{proof}

\begin{remark}
As regards the assumptions on $q$'s in the above corollary, we can just assume $q_0,q_1\in\PP$ when the fixed admissible system is the same for all the spaces.
\end{remark}

Sobolev type embeddings for spaces with variable smoothness and integrability were already studied in \cite{AlmH10} and in \cite{Vyb09}. Embeddings dealing with the more general spaces $\A$ were recently studied in the papers \cite{GonMN14}, \cite{MouNS13} (for constant $q$ in the $B$ case).

\section{Proofs and comments on atomic and molecular results}\label{sec:proofs}

\subsection{Proof of Proposition~\ref{pro:conv}}\label{sec:proof-pro}\hspace{100mm}\\
\emph{Step 1}: Convergence of the inner sum in \eqref{conv} (both pointwisely and in $\cS'$). By the properties of the class $\W$ we have
$$
2^{-j\alpha_1} w_j(2^{-j}m)\, (1+2^j|x-2^{-j}m|)^\alpha\,(1+|x|)^\alpha \gtrsim 2^{-j\alpha_1} w_j(x)\,(1+|x|)^\alpha \gtrsim 1
$$
(with the implicit constants independent of $x\in\Rn$, $j\in\Nz$, $m\in\Z^n$). On the other hand, the inequality
\begin{equation}\label{ineq}
|\lambda_{jm}|\, w_j(2^{-j}m)\,\|\chi_{jm}\,|L_\px\| \le \|\lambda\,| b^\w_{\px,\infty}\|
\end{equation}
holds for every $j\in\Nz$ and $m\in\Z^n$. Hence, for any $\phi\in\cS$, we have

\begin{align*}
& \quad \ \int_\Rn \sum_{m\in\Zn} |\lambda_{jm}|\, |[aa]_{jm}(x)|\,|\phi(x)|\,dx\\
& \lesssim \int_\Rn \sum_{m\in\Zn} |\lambda_{jm}| 2^{-j\alpha_1} w_j(2^{-j}m)\, (1+2^j|x-2^{-j}m|)^{\alpha-M}\,(1+|x|)^\alpha |\phi(x)|\,dx \\
& \lesssim 2^{-j\big(\alpha_1-\frac{n}{p^-}\big)} \, \|\lambda\,| b^\w_{\px,\infty}\| \int_\Rn \sum_{m\in\Zn} (1+2^j|x-2^{-j}m|)^{\alpha-M}\,(1+|x|)^\alpha |\phi(x)|\,dx \\
& \lesssim 2^{-j\big(\alpha_1-\frac{n}{p^-}\big)} \, \|\lambda\,| b^\w_{\px,\infty}\| \int_\Rn (1+|x|)^\alpha |\phi(x)|\,dx
%& \lesssim 2^{-j\big(\alpha_1-\frac{n}{p^-}\big)} \, \|\lambda\,| b^\w_{\px,\infty}\| \, \mathfrak{p}_{n+2+[\alpha]}(\phi) \label{finite}
\end{align*}
where in the second step we used that $2^{j\frac{n}{p^-}} \|\chi_{jm}\,|L_\px\| \gtrsim 1$ (by \eqref{norm-charact-func-cubes}) combined with inequality \eqref{ineq}, and in the third we observed that
$$
\sum_{m\in\Z^n} (1+2^j|x-2^{-j}m|)^{\alpha-M} \lesssim \sum_{m\in\Z^n} (1+|[2^jx]-m|)^{\alpha-M} = \sum_{m'\in\Z^n} (1+|m'|)^{\alpha-M} < \infty
$$
since $M-\alpha>n$ (so the sum is bounded from above by a constant not depending on $x$ nor $j$). Therefore, there exists $c>0$ such that
\begin{equation}\label{finite}
\int_\Rn \sum_{m\in\Zn} |\lambda_{jm}|\, |[aa]_{jm}(x)|\,|\phi(x)|\,dx \le c \, 2^{-j\big(\alpha_1-\frac{n}{p^-}\big)} \, \|\lambda\,| b^\w_{\px,\infty}\| \, \mathfrak{p}_{n+2+[\alpha]}(\phi)
\end{equation}
for all $j\in\Nz$, $\lambda\in b^\w_{\px,\infty}$ and $\phi\in\cS$, where
$$\mathfrak{p}_{N}(\phi):= \sup_{x\in\Rn}(1+|x|)^N \sum_{|\gamma|\le N} |D^\gamma \phi(x)|.$$
Since the right-hand side in \eqref{finite} is finite, we conclude that the inner sum in \eqref{conv}
is absolutely convergent almost everywhere (a.e.), and consequently convergent a.e..\\
It remains to show that this sum converges also in $\cS'$ to the regular distribution given by
\begin{equation}\label{fj-limit}
f_j(x):= \sum_{m\in\Zn} \lambda_{jm}\, [aa]_{jm}(x), \ \ \ x-\textrm{a.e. in} \ \ \Rn.
\end{equation}
The idea here is to reduce the convergence problem in $\cS'$ to the study of convergence of appropriate families in $\mathbb{C}$. Observing that $\Zn$ is (infinitely) countable and that the sum is absolutely convergent, as mentioned above, after re-indexing the terms we can write
$$ f_j= \sum_{k\in\Nz} \lambda_{jk}\, [aa]_{jk} = \sum_{k=0}^\infty \lambda_{jk}\, [aa]_{jk}, \ \ \ \textrm{a.e. in} \ \ \Rn.$$
Given any $\phi\in\cS$, by \eqref{finite} the Lebesgue dominated convergence theorem can be applied to yield
$$
\lim_{\kappa\to\infty} \int_{\Rn} \sum_{k=0}^\kappa \lambda_{jk}\, [aa]_{jk}(x)\, \phi(x)\,dx = \int_\Rn f_j(x)\, \phi(x)\,dx\,,
$$
that is,
$$
\sum_{k=0}^\infty \langle\lambda_{jk}\, [aa]_{jk}, \phi\rangle = \langle f_j, \phi\rangle \ \ \ \textrm{in} \ \ \mathbb{C}.
$$
Clearly we can use the same arguments above whatever the order we choose in $\Nz$ to construct a standard series, so
$$
\sum_{k=0}^\infty \langle\lambda_{jk}\, [aa]_{jk}, \phi\rangle \ \ \textrm{is unconditionally convergent to} \ \ \langle f_j, \phi\rangle \ \ \ \textrm{in} \ \ \mathbb{C}
$$
(for each fixed $j\in\Nz$) and therefore
$$
\sum_{m\in\Z^n} \langle\lambda_{jm}\, [aa]_{jm}, \phi\rangle = \langle f_j, \phi\rangle \ \ \ \textrm{in} \ \ \mathbb{C}.
$$
Since $\phi$ is arbitrary in $\cS$, for any $j\in\Nz$ we finally conclude that the inner sum in \eqref{conv} converges indeed in $\cS'$ (to the regular distribution $f_j$ defined in \eqref{fj-limit}).

\emph{Step 2}: Convergence of the outer sum in \eqref{conv}. This convergence follows if we show that there exist $N\in\N$ and $c>0$ such that
$$
\sum_{j=0}^\infty \big| \langle f_j,\phi\rangle\big| \le c\, \mathfrak{p}_N(\phi)
$$
for all $\phi\in\cS$, where $f_j$ are the regular distributions given by \eqref{fj-limit}. Such inequality can be derived following the arguments from \cite[Appendix]{Kem10} (the choice of $N$ being possible essentially due to the requirements in \eqref{conv-conditions}), provided one uses the convergence of the inner sum not only in the pointwise sense but also in the sense of distributions (cf. Step 1 above). That approach leads us to the explicit sufficient condition on the size of $M$ given in \eqref{conv-conditions}.

We note that a key tool used in the proof over there is the discrete Sobolev embedding
$$
b^{\mathbf{{w}^0}}_{\pzx,\infty}  \hookrightarrow b^{\mathbf{w}^1}_{\pumx,\infty} \ \ \ \ \ \ \ \ \text{for} \ \ \  \ \ \ \frac{w^0_j(x)}{w^1_j(x)} = 2^{j\left(\frac{n}{p_0(x)}-\frac{n}{p_1(x)}\right)}, \ \ x\in\Rn, \ \ j\in\Nz,
$$
where $\mathbf{w}^1\in \W$ and $p_0,p_1\in\PPlog$ with $p_0\le p_1$.

\emph{Step 3}: We show the convergence of the double series in \eqref{conv-double}. Arguing as before, there exist $N\in\N$ and $c>0$ such that
\begin{equation*}
\sum_{j=0}^\infty\sum_{m\in\Zn} |\langle \lambda_{jm}\, [aa]_{jm}, \phi\rangle| \le c \, \mathfrak{p}_{N}(\phi)
\end{equation*}
for all $\phi\in\cS$ . On the other hand, the Fubini's theorem for sums yields
\begin{equation*}
\sum_{(j,m)\in\Nz\times\Zn} |\langle \lambda_{jm}\, [aa]_{jm}, \phi\rangle| = \sum_{n=0}^\infty\sum_{m\in\Zn} |\langle \lambda_{jm}\, [aa]_{jm}, \phi\rangle| \le c \, \mathfrak{p}_{N}(\phi)
\end{equation*}
with $c>0$ independent of $\phi$. Therefore, arguing as in Step 2, we can show that the sum \eqref{conv-double} converges in $\cS'$ to the same distribution given by \eqref{conv}, namely
$$
\sum_{(j,m)\in\Nz\times\Zn} \lambda_{jm}\, [aa]_{jm} = \sum_{j=0}^\infty f_j
$$
with the distributions $f_j$ defined by \eqref{fj-limit}.\hfill$\square$

\subsection{On the proof of Theorem~\ref{thm:decomp1}}\label{sec:proof-decomp1}\hspace{100mm}\\
The proof can be carried out using the general scheme given in \cite{Kem10}, so we are not going to repeat the arguments. However, since modifications to the usual auxiliary results are needed in order to include the more complicated case of variable $q$, mainly in the $B$ case which was not studied in \cite{Kem10}, we make some comments here.

The first point has to do with the bad behavior of the Hardy-Littlewood maximal operator in the mixed spaces $L_\px(\ell_\qx)$ and $\ell_\qx(L_\px)$ when $q$ is variable. Instead of the usual technical estimates (as in \cite[Lemma~3.7]{Kem10}, \cite[Lemma~7.1]{Kyr03}), convolution inequalities involving $\eta$-functions are used in combination with Lemma \ref{lem:conv-eta}. This approach was already suggested in \cite[Section~5]{Kem10} when dealing with the space $\F$, but we find the next formulation, which slightly differs from Lemma~5.5 in that paper, more appropriate in general.
\begin{lemma}\label{lem:eta-instaed-maximal}
Let $j,\nu\in\Nz$, $x\in\Rn$, $0<t\le 1$, $R>n/t$ and $(h_{\nu m})_{m\in\Zn} \subset \mathbb{C}$. Then
$$
\sum_{m\in\Zn} |h_{\nu m}| (1+2^{\min\{\nu,j\}} |x-2^{-\nu}m|)^{-R} \lesssim \max\{1, 2^{(\nu-j)R}\} \left( \eta_{\nu,Rt} \ast \Big| \sum_{m\in\Zn} h_{\nu m}\, \chi_{\nu m} \Big|^t  \right)^{1/t}(x),
$$
with the implicit constant independent of $j,\nu,x$ and $(h_{\nu m})$.
\end{lemma}

Another important tool used in the proof of Theorem~\ref{thm:decomp1} is the discrete convolution inequality \cite[Lemma~9]{KemV12}, which is not so easy to prove in the $B$ case and when $q$ is variable. In this respect, we refer the reader to \cite[Lemma~3.4]{AlmC15a} for a detailed proof of it, generalizing the arguments from \cite{Rych99}.

Finally we point out a further statement needed, now for dealing with convolutions between molecules and functions from an admissible system $\{\varphi_j\}$. It is similar to \cite[Lemma~3.3]{FraJ85}:
\begin{lemma}
Let $\{\varphi_j\}$ be an admissible system and $\big([aa]_{jm}\big)$ be a system of $(K,L,M)$-molecules if $j\in\N$, and $(K,0,M)$-molecules if $j=0$, for given $K,L\in\Nz$ and $M>L+n$. Let also $N\ge 0$ be such that $M>N+L+n$. Then
\[
\big|(\varphi_j \ast [aa]_{\nu m})(x)\big| \lesssim 2^{-(\nu-j)(L+n)} (1+2^j|x-2^{-\nu}m|)^{-N}, \ \ \ \ \text{for} \ \ j\le \nu,
\]
and
\[
\big|(\varphi_j \ast [aa]_{\nu m})(x)\big| \lesssim 2^{-(j-\nu)K} (1+2^\nu|x-2^{-\nu}m|)^{-M}, \ \ \ \ \text{for} \ \ j\ge \nu,
\]
with the implicit constants independent of $x\in\Rn$, $m\in\Zn$ and $j,\nu\in\Nz$, and even independent of the particular system of molecules taken (as long as $K,L,M$ are kept fixed).
\end{lemma}
This is partly used to show that, for each $j\in\Nz$, both sums in
\[
\sum_{\nu=0}^{\infty} \sum_{m\in\Zn} \lambda_{\nu m} \,(\varphi_j \ast [aa]_{\nu m})
\]
converge both pointwisely and in $\cS'$ (to the same regular distribution). Although we are skipping many details here, we would like to stress that both types of convergence are important in a proof of a result like Theorem~\ref{thm:decomp1} following the mentioned scheme, something which is not always clearly stated in the literature.

\subsection{Proof of Theorem~\ref{thm:decomp2}}\label{sec:proof-decomp2}\hspace{100mm}\\
Let us first formulate an auxiliary  result, which is a refinement of \cite[Lemma~3.6]{Kem10} (see also the proof of \cite[Theorem~2.6.ii)]{FraJ85}):
\begin{lemma}\label{formula}
Given $L\in\Nz$ and $\sigma>0$, there are $\delta>0$ and $\phi_0,\phi,\psi_0,\psi \in \cS$ such that
\begin{align*}
& \supp \phi_0, \; \supp \phi \subset B(0,\sigma), \ \ \ \ \int_{\Rn}x^\beta \phi(x)\,dx = 0 \ \ \text{for} \ \ 0\le |\beta|<L,\\
&  |\hat{\phi}_0(x)| >0 \ \ \text{for} \ \ |x|\le 2\delta, \ \ \ \ |\hat{\phi}(x)| >0 \ \ \text{for} \ \ \delta/2 \le |x|\le 2\delta,\\
& |\hat{\psi}_0(x)| >0 \ \ \text{iff} \ \ |x|< 2\delta, \ \ \ \ |\hat{\psi}(x)| >0 \ \ \text{iff} \ \ \delta/2 < |x| < 2\delta,
\end{align*}
and, for any $f\in\cS'$, there holds
\[
\hat{\phi}_0 \hat{\psi}_0 f + \sum_{j=1}^\infty \hat{\phi}(2^{-j}\cdot)\, \hat{\psi}(2^{-j}\cdot) f = f \ \ \ \ \ (\text{convergence in} \ \ \cS').
\]
\end{lemma}

\emph{Step 1}: Proof of assertion (a). Let $f\in\cS'$. For the given $L$ and $\sigma:=\tfrac{1}{2}(d-1)>0$, consider $\delta, \phi_0, \phi,\psi_0,\psi$ according to the previous lemma and apply the corresponding formula to $\hat{f}\in\cS'$. Then
\begin{equation}\label{formula-applied}
(2\pi)^n f=\phi_0 \ast\psi_0\ast f + \sum_{j=1}^\infty \phi_j \ast\psi_j \ast f \ \ \ \ \ (\text{convergence in} \ \ \cS'),
\end{equation}
with $\phi_j=2^{jn}\phi(2^j\cdot)$ and $\psi_j=2^{jn}\psi(2^j\cdot)$, $j\in\N$. Simple calculations show that, for fixed $j\in\Nz$ and $x\in\Rn$, we get
\[
(\phi_j \ast\psi_j \ast f)(x)= \sum_{m\in\Zn} \int_{Q_{jm}}\phi_j(x-y)(\psi_j\ast f)(y)\,dy =(2\pi)^n \sum_{m\in\Zn}\lambda_{jm}\,a_{jm}(x)
\]
where
\[
\lambda_{jm} := C_{K,\Phi} \, \sup_{y\in Q_{jm}} |(\psi_j \ast f)(y)| \ \ \ \ \ \text{with} \ \ \ C_{K,\Phi}:=\max_{\Phi\in\{\phi_0,\phi\} \atop |\beta|\le K} \sup_{y\in\Rn} |D^\beta \Phi(y)|
\]
and
\[
a_{jm}(x):= \frac{(2\pi)^{-n}}{\lambda_{jm}} \int_{Q_{jm}}\phi_j(x-y)(\psi_j\ast f)(y)\,dy \ \ \ \ \ \text{if} \ \ \lambda_{jm} \neq 0,
\]
otherwise we take $a_{jm}(x) \equiv 0$.\\
After some calculations with the auxiliary  functions from Lemma~\ref{formula} we conclude that $a_{0m}$ and $a_{jm}$, $j\in\N$, are indeed $(K,0,d)$-atoms  and $(K,L,d)$-atoms, respectively, and, moreover, they are Schwartz functions (note that they have continuous derivatives of all orders).\\
Finally we note that from \eqref{formula-applied} and the calculations above we  have
$$
f=\sum_{j=0}^\infty \sum_{m\in\Zn} \lambda_{jm} a_{jm}
$$
with the outer sum converging in $\cS'$ and the inner sum taken in the pointwise sense.

\emph{Step 2}: Proof of assertion (b).  Let $f\in\A$. We show that the resulting $\lambda(f):=(\lambda_{jm})$ as defined above is in the corresponding sequence space $\as$ and its quasi-norm is controlled from above by a constant times $\|f\,|\,\A\|$.\\
For each $j\in\Nz$ and $x\in \cup_{m\in\Zn} \mathring{Q}_{jm}$, we have
\[
\big| \sum_{m\in\Zn} \lambda_{jm}\chi_{jm}(x)\big| \le C_{K,\Phi} \, \sup_{y\in Q_{jm}} |(\psi_j \ast f)(y)|,
\]
where the $m$ on the right-hand side is the only $m\in\Zn$ such that $x\in Q_{jm}$. Multiplying both sides of the inequality above by $w_j(x)$ and noting that $2^j|x-y|\le \sqrt{n}$ when $x$ and $y$ belong to the same cube $Q_{jm}$, we get
\[
\big| \sum_{m\in\Zn} w_j(x) \lambda_{jm}\chi_{jm}(x)\big| \le (1+n^{\tau/2})\,C_{K,\Phi} \,w_j(x) \big( \psi^\ast_j f\big)_\tau(x),
\]
almost everywhere, where
\[
\big( \psi^\ast_j f\big)_\tau(x):= \sup_{y\in\Rn} \frac{|\psi_j\ast f(y)|}{1+|2^j(x-y)|^{\tau}}
\]
stands for the usual Peetre maximal functions (here $\tau>0$ is arbitrary). It is easy to check that the functions $\psi_0,\psi$ above fit in the requirements of \cite[Theorem~3.1]{AlmC15a} if we pick a positive $\varepsilon$ there appropriate to the previously fixed $\delta>0$, namely taking $\varepsilon\in(\delta,2\delta)$. If we are careful in choosing $\tau>0$ large enough according to the same theorem, then we obtain
\[
  \|\lambda(f)\,| \bs\| \lesssim  \|\big(w_j\, (\psi^\ast_j f)_\tau\big)_j\,| \ell_\qx(L_\px)\| \approx \|f\,| \B\| < \infty
\]
in the $B$ case, and similarly in the $F$ case. By an argument similar to the one used in Step~1 of the proof of Proposition~\ref{pro:conv} (see Subsection~\ref{sec:proof-pro}), now adapted to our situation dealing with atoms, we can show that, with no further assumptions, the inner sum above also converges in $\cS'$ to the regular distribution given by the corresponding pointwise sum. The proof is then complete.\hfill $\square$

\subsection*{Acknowledgement} We would like to thank Henning Kempka for a discussion about the proof of Lemma~\ref{lem:eta-instaed-maximal} and on the size conditions on $M$ referred in Remark~\ref{rem:atoms-also}.

%%%%%%%%%%%%%%%%%%%%%%%%%%%%%%%%%%%%%%%%%%%%%%%%%

%%%%%%%%%%%%%%%%%%%%%%%%%%%%%%%%%%%%%%%%%%%%%%%%%
%%%%%%%%%%%%%%%%%%%%%%%%%%%%%%%%%%%%%%%%%%%%%%%%%
%%%%%%%%%%%%%%%%%%%%%%%%%%%%%%%%%%%%%%%%%%%%%%%%%

\end{document}